\documentclass[11pt]{article}
\author{Peter J. Brockwell \thanks{Departments of Statistics, Columbia University, New York and Colorado State University, Fort Collins,
USA,
email: pjb2141@columbia.edu }
\and
Vincenzo Ferrazzano
\thanks{Center for Mathematical Sciences, Technische Universit\"at M\"unchen, Boltzmannstr. 3, 85748 Garching b. M\"unchen, Germany,
email: ferrazzano@ma.tum.de, http://www-m4.ma.tum.de/pers/ferrazzano/}
\and
Claudia Kl\"uppelberg\thanks{Center for Mathematical Sciences, and Institute for Advanced Study, Technische Universit\"at M\"unchen, Boltzmannstr. 3, 85748 Garching b. M\"unchen, Germany,
email: cklu@ma.tum.de, http://www-m4.ma.tum.de}
}

\title{High-frequency sampling and kernel estimation for continuous-time moving average processes}

\usepackage{natbib}
\usepackage[latin9]{inputenc}
\usepackage[english]{babel}
\usepackage[T1]{fontenc}
\usepackage{amsmath,amssymb}
\usepackage{lscape}
\usepackage[final]{showkeys}
\usepackage{graphicx}
\usepackage{url}
\usepackage{latexsym}
\usepackage[usenames]{color}
\usepackage{float}
\usepackage{epsfig}
\usepackage{amssymb}
\usepackage{amsfonts}
\usepackage{amsmath}
\usepackage{amsthm}

\numberwithin{equation}{section}

\newtheorem{thm}{Theorem}[section]
\newtheorem{cor}[thm]{Corollary}

\newtheorem{lem}[thm]{Lemma}
\newtheorem{prop}[thm]{Proposition}
\newtheorem{defn}[thm]{Definition}
\newtheorem{example}[thm]{Example}
\newtheorem{oss}[thm]{Remark}
\newtheorem{fig}[thm]{Figure}

\newcommand{\displayfrac}[2]{\frac{\displaystyle #1}{\displaystyle #2}}

\newcommand{\bthe}{\begin{thm}}
\newcommand{\ethe}{\end{thm}}

\newcommand{\ble}{\begin{lem}}
\newcommand{\ele}{\end{lem}}

\newcommand{\bde}{\begin{defn}}
\newcommand{\ede}{\end{defn}}

\newcommand{\bco}{\begin{cor}}
\newcommand{\eco}{\end{cor}}

\newcommand{\bpr}{\begin{prop}}
\newcommand{\epr}{\end{prop}}

\newcommand{\bproof}{\begin{proof}}
\newcommand{\eproof}{\end{proof}}

\newcommand{\bexam}{\begin{example}\rm}
\newcommand{\eexam}{\halmos\end{example}}

\newcommand{\brem}{\begin{oss}\rm}
\newcommand{\erem}{\halmos\end{oss}}

\newcommand{\bfi}{\begin{fig}}
\newcommand{\efi}{\end{fig}}

\newcommand{\beao}{\begin{eqnarray*}}
\newcommand{\eeao}{\end{eqnarray*}\noindent}

\newcommand{\beam}{\begin{eqnarray}}
\newcommand{\eeam}{\end{eqnarray}\noindent}

\newcommand{\barr}{\begin{array}}
\newcommand{\earr}{\end{array}}

\newcommand{\beq}{\begin{equation}}
\newcommand{\eeq}{\end{equation}}

\def\bbr{{\Bbb R}}

\def\bbc{{\Bbb C}}
\def\bbz{{\Bbb Z}}
\def\bbe{{\Bbb E}}

\def\calr{{\mathcal{R}}}

\newcommand{\dsum}{\displaystyle\sum}

\newcommand{\al}{{\alpha}}

\newcommand{\la}{{\lambda}}

\newcommand{\si}{{\sigma}}

\newcommand{\CARMA}{{\rm CARMA}}
\newcommand{\FICARMA}{{\rm FICARMA}}

\newcommand{\AR}{{\rm AR}}

\newcommand{\CMA}{{\rm CMA}}

\newcommand{\ov}{\overline}
\newcommand{\wh}{\widehat}

\def\halmos{\hfill $\Box$  \medskip }

\begin{document}
\maketitle

\begin{abstract}
Interest in continuous-time processes has increased rapidly in recent years, largely because of high-frequency data available in many applications.  We develop a method for estimating the kernel function $g$ of a second-order stationary L\'evy-driven continuous-time moving average (CMA) process $Y$ based on observations of the discrete-time process $Y^\Delta$ obtained by sampling $Y$ at $\Delta, 2\Delta,\ldots,n\Delta$ for small $\Delta$.  We approximate $g$ by $g^\Delta$ based on the Wold representation  and prove its pointwise convergence to $g$ as $\Delta\rightarrow 0$ for $\CARMA(p,q)$ processes.  Two non-parametric estimators of $g^\Delta$, based on the innovations algorithm and the Durbin-Levinson algorithm, are proposed to estimate $g$. For a Gaussian CARMA process we give conditions on the sample size $n$ and the grid-spacing $\Delta(n)$ under which the innovations estimator is consistent and asymptotically normal as $n\rightarrow\infty$.
The estimators can be calculated from sampled observations of {\it any} CMA process and simulations suggest that they perform well even outside the class of CARMA processes.  
We illustrate their performance for simulated data
and apply them to the Brookhaven turbulent wind speed data.  
Finally we extend results of \citet{bfk:2011:1} for sampled CARMA processes to a much wider class of CMA processes.
\end{abstract}
\bigskip

\noindent
\begin{tabbing}
\end{tabbing}

\medskip


\newpage

\section{Introduction}\label{s1}

We are concerned in this paper with causal continuous-time moving averages of the form
\beam\label{CMA}
Y_t:=\int_{-\infty}^\infty g(t-s) dL_s,\quad t\in\bbr,
\eeam
where $\{L_t\}_{t\in\bbr}$ is a L\'evy process with $EL_1=0$ and $EL_1^2=\sigma^2<\infty$. The kernel function $g$ is assumed to be square integrable, zero on $(-\infty,0]$
(for causality) and such that the Fourier transform ${\tilde g}(\omega):=\int_{-\infty}^\infty e^{it\omega}g(t)dt$ is non-zero if the imaginary part of $\omega$ is
greater than or equal to $0$.  (This {\it minimum phase} property is the continuous-time analogue of the discrete-time notion of invertibility.)
The process {$Y$} defined by (\ref{CMA}) is then a zero-mean strictly and weakly stationary  process.  For the estimation of $g$ discussed in Sections 4 and 5 we make the additional assumption that $L$ has been standardized so that $\sigma^2=1$, since otherwise  $g$ is identifiable only to within multiplication by a constant.  In the special case when $L$ is Brownian motion the process
(\ref{CMA}) is Gaussian, but by allowing $L$ to be a second-order L\'evy process we greatly expand the class of possible marginal distributions for $Y_t$.

The integral in \eqref{CMA} is understood in the $L^2$-sense and, since we use only second-order properties in our analysis, the results
apply more generally to processes defined by (\ref{CMA}) when $L$ is a process with stationary orthogonal increments such that $EL_1=0$ and $EL_1^2=\sigma^2$
as in  \cite{Doob:1990df} Ch. IX. It is important to note however that when $L$ is a given L\'evy process, $Y$ is completely characterized by $g$, while the spectral density of $Y$ characterizes only the second-order properties.  Throughout this paper, stationarity will always mean weak stationarity.

Examples of \CMA\ processes are the Ornstein-Uhlenbeck process, with $g(t)=e^{\lambda t}{\bf 1}_{(0,\infty)}(t)$, where $\lambda<0$, and the more general class of continuous-time autoregressive moving average (\CARMA) processes studied by \citet{doob:1944} for Gaussian $L$.  State-space representations of these processes were exploited by  \cite{Jones2} and \cite{Jones1} for dealing with missing values in time series, and by \cite{brockwell3} for the study of L\'evy-driven  CARMA processes. Long-memory versions have been developed by \cite{BroMar05} and \cite{Marquardt:2006}.  $\CMA$ processes, and in particular CARMA processes,  constitute a very large class of continuous-time stationary processes, the latter playing the same role in continuous time as do the ARMA processes in discrete time.  Estimation for continuous-time processes has been considered from various points of view by many authors
including  Phillips (1959),  Pham (1977), Robinson (1977)  and Bergstrom (1990).

The present paper was motivated by a study of the Brookhaven turbulence data (see \citet{techEnzo} for a detailed description and further references).
The data consist of twenty million values, sampled at 5000Hz (i.e. 5000 values per second) over a time interval of 4000 seconds.   One of the goals was
to estimate the kernel $g$ in a model of the form (\ref{CMA}) for the underlying continuous-time process from which the observations were sampled.  It is intuitively clear
that the kernel $g$ should be closely related to the coefficients $\psi_j^\Delta$ in the Wold representation,
\begin{equation}\label{Wold1}
Y^\Delta_n=\sum_{j=0}^\infty \psi_j^\Delta Z_{n-j}^\Delta,\quad n\in\bbz,\quad\{Z_n^\Delta\}_{n\in\bbz}\sim{\rm WN}(0,\sigma^2_\Delta),
\end{equation}
of
the sampled process $Y^\Delta=\{Y_{n\Delta}\}_{n\in\mathbb{Z}}$ when the temporal spacing $\Delta$ of the observations is very small.  In Section 3 we make this connection
precise for the class of CARMA processes by showing that, as $\Delta\rightarrow 0$,  the function
\begin{equation}\label{gdeltadef1}
g^\Delta(t):=\sum_{j=0}^\infty{{\sigma_\Delta}\over{\sqrt{\Delta}}}\psi^\Delta_j{\bf 1}_{[j\Delta,(j+1)\Delta)}(t)\end{equation}
converges pointwise to
$\sigma g$ (or to $g$ if $L$ is standardized so that $EL_1^2=1$).  Proof of this result requires a more precise analysis of the relation between the continuous-time
process $Y$ and the sampled sequence $Y^\Delta$ than the one given in  \citet{bfk:2011:1}.  The proposed estimators of $g$ are non-parametric estimators of $g^\Delta$
obtained by either Durbin-Levinson or innovations algorithm estimation of $\psi_j^\Delta$ and $\sigma_\Delta$.  In the important Gaussian case consistency of the innovations estimator is established as Corollary~\ref{cor:const} and a central limit theorem is given as Theorem~\ref{clt}.

Although our analysis focuses on CARMA processes, the estimation algorithms can be applied to the sampled observations of {\it any} CMA process of the form (\ref{CMA}) since they involve only the estimation of $\psi_j^\Delta$ and $\sigma_\Delta$ with $\Delta$ small.  In Section 4 we illustrate the performance of the estimators using simulated data and in Section 5
we apply them to the Brookhaven data discussed above.  The outcome of a detailed statistical analysis of turbulence data is presented in~\citet{fk:2011:1}.

In Section 6 we extend the asymptotic results of  \citet{bfk:2011:1} to a broader class of CMA processes,
which includes fractionally integrated CARMA processes, and discuss their implications for local sample-path behaviour.

We use the following notation throughout:
$\Re(z)$ denotes the real part of the complex number $z$; $B$ denotes the backward shift operator, $BY^\Delta_n:=Y^\Delta_{n-1}$ for $n\in\bbz$.
$a(\Delta)\sim b(\Delta)$ means $\lim_{\Delta\rightarrow 0} a(\Delta)/b(\Delta) =1$;

\section{The sampled sequence $Y^\Delta$ when $Y$ is a $\CARMA(p,q)$ process}
For non-negative integers $p$ and $q$ such that $q<p$, a $\CARMA(p,q)$ process $Y=\{Y_t\}_{{t\in\mathbb{R}}}$, with
coefficients
$a_{1},\ldots,a_{p}$, $b_{0},\ldots,b_{q}\in\bbr$, and driving L\'evy process $L$,  is defined to be a
strictly stationary solution of the
suitably interpreted formal equation,
\begin{equation}\label{1.1}
a(D)Y_t=b(D)DL_t,\quad t\in\mathbb{R},
\end{equation}
where $D$ denotes differentiation with respect to $t$, $a(\cdot)$ and $b(\cdot)$ are the polynomials,
$$a(z):=z^{p}+a_{1}z^{p-1}+\cdots+a_{p}\quad\mbox{and}\quad
b(z):=b_{0}+b_{1}z+\cdots+b_{p-1}z^{p-1},$$
and the
coefficients $b_{j}$ satisfy $b_{q}=1$ and $b_{j}=0$ for $q<j<p$.
The polynomials $a(\cdot)$ and $b(\cdot)$ are assumed to have no common zeroes and, for $Y$ to be causal and minimum phase, their zeroes all
lie in the open left half plane.   The L\'evy process $L$ is assumed throughout to satisfy $EL_t=0$ and Var$(L_t)=\sigma^2|t|$ for all $t\in\mathbb{R}$.

 The kernel (see  \cite{BrLi})  is
\begin{equation}\label{CARMAkernel1}
g(t)={{1}\over{2\pi i}}\int_\rho {{b(z)}\over{a(z)}}e^{tz}dz~{\bf 1}_{(0,\infty)}(t)
=\sum_\lambda{\rm Res}_{z=\lambda}\left(e^{tz}{{b(z)}\over{a(z)}}\right){\bf 1}_{(0,\infty)}(t),\end{equation}
where
the integration is anticlockwise around any simple closed curve $\rho$ in the interior of the left half of the complex plane, encircling the distinct zeroes $\lambda$ of $a(z)$, and
${\rm Res}_{z=\lambda}(f(z))$ denotes the residue of the function $f$ at $\lambda$.

The spectral density $f_\Delta$ of the sampled process $Y^\Delta$ is (see \citet{bfk:2011:1}, Section 2)
\begin{equation}\label{sampledCARMAspec}f_{\Delta}(\omega)={{-\sigma^2}\over{4\pi^2i}}\int_\rho {{b(z)b(-z)}\over{a(z)a(-z)}}{{\sinh(\Delta z)}\over{\cosh(\Delta z)-\cos(\omega)}}dz,~-\pi\le\omega\le \pi,\end{equation}where the integral, as in (\ref{CARMAkernel1}), is anticlockwise around any simple closed curve $\rho$  in the interior of the left half of the complex plane, enclosing the zeroes of $a(z)$.  It is well-known (see e.g.  \cite{BrLi}) that the sampled process $Y^\Delta$ satisfies the ARMA equations,
\begin{equation}\phi^\Delta(B)Y^\Delta_n=\theta^\Delta(B)Z^\Delta_n, ~n\in\mathbb{Z},\label{ARMA}\end{equation}where
$B$ is the backward shift operator,  $\phi^\Delta(z)$ is the polynomial,
\begin{equation}\label{ARpol}\phi^\Delta(z)=\prod_{j=1}^p(1-e^{\lambda_j\Delta}z),\quad z\in\bbc,\end{equation}
$\lambda_1,\ldots,\lambda_p$ are the zeroes of the polynomial $a(z)$,
$\theta^\Delta(z)$ is a polynomial of degree less than $p$, whose zeroes can be chosen, by the minimum-phase assumption on $Y$, to lie in the exterior of the unit disc, and $(Z^\Delta_n)_{n\in\mathbb{Z}}$ is an uncorrelated sequence of zero-mean random variables with variance which we shall denote by $\sigma^2_\Delta$.  The Wold representation of the sampled process is then given by (\ref{Wold1}) with $\psi^\Delta_j$ the coefficient of $z^j$ in the power-series expansion,
\begin{equation}\label{PSexpansion}\psi^\Delta(z):=\sum_{j=0}^\infty \psi^\Delta_j z^j={{\theta^\Delta(z)}\over{\phi^\Delta(z)}}\end{equation}

Although, for any given CARMA process it is a trivial matter to write down the autoregressive polynomial $\phi^\Delta(z)$, it is a much more difficult problem to determine
$\theta^\Delta(z)$.
Although $\theta^\Delta(z)$ cannot be written down explicitly (except in the special case when $Y$ is a CARMA(2,1) process) its asymptotic behaviour, as $\Delta\rightarrow 0$ is specified by the following theorem which is proved in the Appendix.  This theorem, with (\ref{ARpol}) and (\ref{PSexpansion}) will determine the asymptotic behaviour of $\psi_j^\Delta$ and $\sigma_\Delta$ as $\Delta\rightarrow 0$, thereby enabling us to establish the pointwise convergence of the Wold approximation $g^\Delta$ in (\ref{gdeltadef1})  to the kernel $g$.

Before stating the theorem we need to introduce some notation.  We first write the autoregressive and moving average CARMA polynomials in \eqref{1.1} as
\begin{equation}a(z)=\prod_{i=1}^p(z-\lambda_i)~~{\rm and}~~b(z)=\prod_{i=1}^q(z-\mu_i),\end{equation}
and define \begin{equation}\label{zeta1}\zeta_k=1+\mu_k\Delta, ~~k=1,\ldots,q.\end{equation}  The coefficients $\xi_i, ~i=1,\ldots,p-q-1$, are defined to be the zeroes of the polynomial
$\alpha(x)$, the coefficient of $z^{2(p-q)-1}$ in the expansion of $\sinh(z)/(\cosh(z)-1+x) $ in powers of $z$.  Finally we define
\begin{equation}\label{eta1}\eta(\xi_i) := \xi_i-1\pm\sqrt{(\xi_i-1)^2-1},~~i=1,\ldots,p-q-1,\end{equation} with the sign chosen so that $|\eta(\xi_i)|<1$.  We can now state the theorem.

\begin{thm}\label{MArepresentation1}
The moving average process $X_n:=\theta^\Delta(B)Z^\Delta_n$ has the asymptotic representation, as $\Delta\rightarrow 0$,
\begin{equation}\label{MArep}X_n=\prod_{i=1}^{p-1-q}(1+(\eta(\xi_i)+o(1))B)\prod_{k=1}^q(1-(\zeta_k+o(\Delta))B)Z^\Delta_n,~~
\{Z^\Delta_n\}_{n\in\bbz}\sim{\rm WN}(0,\sigma^2_\Delta),\end{equation}where
\begin{equation}\sigma^2_\Delta={{\Delta^{2(p-q)-1}e^{-a_1\Delta}\sigma^2}\over{[2(p-q)-1]!\prod_{i=1}^{p-q-1}\eta(\xi_i)\prod_{k=1}^q\zeta_k}}(1+o(1)),\end{equation}with  $\zeta_k$ and $\eta(\xi_i)$  defined as in (\ref{zeta1}) and (\ref{eta1}).\end{thm}

\brem{\rm (i) The parameters $\zeta_k$ and $\eta(\xi_i)$ may be complex but the moving average operator will have real coefficients because of the
existence of corresponding complex conjugate parameters in the product. \\[2mm]
(ii) \,  The representation in Theorem \ref{MArepresentation1} is a substantial generalization of the one in Corollary 2 of \citet{bfk:2011:1}, since it is not only of higher-order in $\Delta$, but it applies to all $\CARMA(p,q)$ processes, not only to those with $p-q\le 3$.}
\erem

\section{{The convergence of $g^\Delta(t)$ to $\sigma g(t)$ for $\CARMA(p,q)$ processes}}\label{sec:scaling}

In this section we establish the pointwise convergence, as $\Delta\rightarrow 0$, of the Wold approximation $g^\Delta$ to $\sigma g$ when $Y$ is the $\CARMA(p,q)$ process \eqref{1.1}.
(This means that if $L$ is standardized so that $\sigma^2:=EL_1^2=1$ then $g^\Delta$ converges to $g$.)  Before stating the general result,
we illustrate the convergence in the simplest case, namely when $Y$ is a CARMA(1,0) (or stationary Ornstein-Uhlenbeck) process, for which the quantities $\psi^\Delta_j$ and $\sigma_\Delta$
can easily be found explicitly.  The example also illustrates the role of the scale factor $\sigma_\Delta/\sqrt{\Delta}$ which multiplies the Wold coefficients,  $\psi^\Delta_j$
in (\ref{gdeltadef1}).

\bexam[The CARMA(1,0) process]\label{Carma10ex} {This a special case of (\ref{CMA}) with kernel
$$g(t)= e^{\lambda t}{\bf 1}_{(0,\infty)}(t) ~{\rm where}~\lambda<0.$$
The sampled process $Y^\Delta$ is the discrete-time AR(1) process satisfying
$$Y^\Delta_n=e^{\lambda\Delta}Y^\Delta_{n-1}+Z_n^\Delta,\quad n\in\bbz,$$
where (by Lemma~2.1 of \cite{BrLi}) $Z^\Delta=\left\{Z^\Delta_n\right\}_{n\in\bbz}$ is the i.i.d. sequence defined by
$$Z^\Delta_n=\int_{(n-1)\Delta}^{n\Delta}e^{\lambda(n\Delta-u)}dL_u,\quad n\in\bbz.$$In this case it is easy to write down the coefficients $\psi^\Delta_j$ and the white noise variance $\sigma_\Delta^2$ in the Wold representation of $Y^\Delta$.  From well-known properties of discrete-time AR(1) processes, they are $\psi^\Delta_j=e^{j\lambda\Delta}, j=0,1,2,\ldots$, and $\sigma_\Delta^2={{\sigma^2}\over{2\lambda}}(e^{2\lambda\Delta}-1)$.  Substituting these values in the
definition (\ref{gdeltadef1}) we find that
$$g^\Delta(t)=\sum_{j=0}^\infty \sigma\sqrt{{{e^{2\lambda\Delta}-1}\over{2\lambda\Delta}}}
e^{j\lambda\Delta}{\bf 1}_{[j\Delta,(j+1)\Delta)}(t),$$
which converges pointwise to $\sigma g$ as $\Delta\rightarrow 0$.
}
\eexam

The function $g^\Delta$ as specified in (\ref{gdeltadef1}) is well-defined for {\it any} CMA process (\ref{CMA}), and for the corresponding sampled process $Y^\Delta$ there are standard methods for estimating
the Wold parameters $\psi_j^\Delta$ and $\sigma_\Delta$ and hence $g^\Delta$ itself.   For such an estimator of $g^\Delta$  to be useful there are two issues to be considered.  The first is that $g^\Delta$ should be a close approximation to $g$ when $\Delta$ is sufficiently small and the second
concerns the {\it estimation} of $g^\Delta$ from  observations of $Y^\Delta$.  In this section we deal with the first issue by showing that, at least for all CARMA processes,
$g^\Delta$  converges pointwise to
$\sigma g$ as $\Delta\rightarrow 0$. We give the proof
under the assumption that the zeroes $\lambda_1,\ldots,\lambda_p$ of the autoregressive polynomial $a(z)$ all have multiplicity one.
Multiple roots can be handled by supposing them to be separated and letting the separation(s) converge to zero.

The kernel (\ref{CARMAkernel1}) of a causal $\CARMA(p,q)$ process $Y$ whose autoregressive roots each have multiplicity one reduces (see e.g. \cite{BrLi}) to
\begin{equation}\label{CARMAkernel} g(t)=\sum_{j=1}^p{{b(\lambda_j)}\over{a'(\lambda_j)}}e^{\lambda_jt}{\bf 1}_{(0,\infty)}(t), \end{equation}where
$a(z)=\prod_{i=1}^p(z-\lambda_i)$ and $b(z)=\prod_{i=1}^q(z-\mu_i)$ are the autoregressive and moving average polynomials respectively and $a'$ denotes the derivative of the function $a$.  The pointwise convergence of $g^\Delta$ to $g$ is established in the following theorem.

\begin{thm}\label{Woldapprox}
If $Y$ is the $\CARMA(p,q)$ process with kernel (\ref{CARMAkernel}),

(i) the Wold coefficients and white noise variance of the sampled process
$Y^\Delta$ are
\begin{equation}\label{psijdelta}\psi_j^\Delta=\sum_{r=1}^p{{\prod_{i=1}^{p-1-q}(1+(\eta(\xi_i)+o(1))e^{-\lambda_r\Delta})\prod_{k=1}^q(1-(\zeta_k+o(\Delta))e^{-\lambda_r\Delta})}\over{\prod_{m\ne r}(1-e^{(\lambda_m-\lambda_r)\Delta})}}e^{j\lambda_r\Delta},\end{equation}and
\begin{equation}\label{sigmasquared}\sigma^2_\Delta={{\Delta^{2(p-q)-1}e^{-a_1\Delta}\sigma^2}\over{[2(p-q)-1]!\prod_{i=1}^{p-q-1}\eta(\xi_i)\prod_{k=1}^q\zeta_j}}(1+o(1)),\end{equation}with  $\zeta_k$ and $\eta(\xi_i)$ as in (\ref{zeta1}) and (\ref{eta1}) and

(ii) the approximation $g^\Delta$ defined by (\ref{gdeltadef1}) with $\psi_j^\Delta$ and $\sigma^2_\Delta$ as in (\ref{psijdelta}) and (\ref{sigmasquared}) converges pointwise to $\sigma g$  with $g$ as in (\ref{CARMAkernel}).
\end{thm}
\begin{proof} (i) The expression for $\sigma^2_\Delta$ was found already as part of Theorem~\ref{MArepresentation1}.  The coefficient $\psi_j^\Delta$ is the coefficient of $z^j$ in the power series expansion,
$$\sum_{j=0}^\infty\psi_j^\Delta z^j= {{ \prod_{i=1}^{p-1-q}(1+(\eta(\xi_i)+o(1))z)\prod_{k=1}^q(1-(\zeta_k+o(\Delta))z)}\over{\prod_{m=1}^p(1-e^{\lambda_m\Delta} z)}},             $$which can be seen, by partial fraction expansion, to be equal to (\ref{psijdelta}).

\vskip .1in(ii) The convergence of $g^\Delta$ follows by substituting $\psi_j^\Delta$ and $\sigma^2_\Delta$ from (\ref{psijdelta}) and (\ref{sigmasquared}) into (\ref{gdeltadef1}), substituting for $\zeta_k$ from  (\ref{zeta1}), letting $\Delta\rightarrow 0$ and using the identities
$$a'(\lambda_r)=\prod_{m\ne r}(\lambda_r-\lambda_m)$$and
$$\prod_{i=1}^{p-q-1}{{(1+\eta(\xi_i))^2}\over{\eta(\xi_i)}}=\prod_{i=1}^{p-q-1}{{\xi_i}\over{2}}=[2(p-q)-1].$$

\vskip -.2in\end{proof}

\brem\rm
Although we have established the convergence of $g^\Delta$ only for CARMA processes, we conjecture that it holds for all processes defined as in (\ref{CMA}).  In practice we have found that estimation of $g$ by non-parametric estimation of $g^\Delta$ with $\Delta$ small works well not only for CARMA processes but also  for simulated processes with non-rational spectral densities.
\erem

\section{Estimation of $g^\Delta$}\label{sec:approx}

Given observations of $Y^\Delta$ with $\Delta$ small, we estimate the kernel $g$ by estimating the approximation $g^\Delta$ defined in (\ref{gdeltadef1}) which, as shown in the preceding section, converges pointwise to $\sigma g$ as $\Delta\rightarrow 0$ for all $\CARMA(p,q)$ processes.
If the driving  L\'evy process is standardized so that Var($L_1$)=1, then $g^\Delta$ converges pointwise to $g$.   From now on we make this assumption since without it $g$ is identifiable only to within multiplication by a constant.

To estimate $g^\Delta$ it suffices to estimate the coefficients and white noise variance in the
Wold representation (\ref{Wold1}) of $Y^\Delta$, for which standard non-parametric  methods are available.
Being non-parametric they require no {\it a priori} knowledge of the order
of the underlying CARMA process and moreover they can be applied to the sampled observations of any CMA of the form \eqref{CMA}.
The most direct estimator of the Wold parameters of a causal invertible ARMA process is based on the innovations algorithm (see Brockwell and Davis (1991), Section 8.3).
Noting that the definition (\ref{gdeltadef1}) is equivalent to
\begin{equation}\label{kernel:estimation:formula}g^\Delta(t)={{\sigma_\Delta}\over{\sqrt{\Delta}}}\psi^\Delta_{\lfloor t/\Delta\rfloor},\end{equation}
where $\lfloor t/\Delta\rfloor$ denotes the integer part of $t/\Delta$,
we obtain the following asymptotic result for the estimation of $g^\Delta(t)$ for fixed $\Delta$ as $n\to\infty$ in the important cases when $Y$ is either a Gaussian CARMA process
of arbitrary order or a CARMA(1,0) process with arbitrary second-order driving L\'evy process.
It follows directly from Theorem~2.1 of~\cite{bd:1988}.

\begin{thm}\label{consistency}
Suppose that $Y$ is a  Gaussian  $\CARMA(p,q)$ process or a general L\'evy-driven $\CARMA(1,0)$ process observed at times
$k\Delta,~k=1,\ldots,n$.  For any fixed $t\ge0$ and $\Delta>0$,  let $r=\lfloor t/\Delta\rfloor$.
Then the innovations estimators ${\hat\theta}_{m,r}$ and ${\hat v}_m$ of  $\psi^\Delta_{\lfloor t/\Delta\rfloor}$ and $\sigma_\Delta^2$, respectively, have the following asymptotic properties.
For any sequence of positive integers $\{m(n),  n=1,2,\ldots\}$ such that $m<n$, $m\rightarrow\infty$ and $m=o(n^{1/3})$ as $n\rightarrow\infty$,
${\hat\theta}_{m(n),r}$ is consistent for $\psi^\Delta_{\lfloor t/\Delta\rfloor}$ and asymptotically normal.  More specifically,
\begin{equation}\label{innovations:algo:as}
{\sqrt{n}}({\hat\theta}_{m(n),r}-\psi^\Delta_{\lfloor t/\Delta\rfloor}) \, \Rightarrow \, {\rm N}(0,a_\Delta),\end{equation}
where $a_\Delta:=\dsum_{j=0}^{r-1}(\psi_j^\Delta)^2$, and
\begin{equation}\label{innovations:algo:cons}{\hat v}_{m(n)} \rightarrow\hskip -.16in ^P \hskip .1in\sigma_\Delta^2.\end{equation}
\end{thm}

\brem
(i) \, The restriction to either Gaussian or CARMA(1,0) processes stems from the fact that in these cases the driving noise sequence $\{Z^\Delta_n\}_{n\in\bbz}$ is i.i.d. as required by Theorem~2.1 of \cite{bd:1988}.  By Lemma~2.1 of \cite{BrLi} the driving noise sequence is in general uncorrelated but not i.i.d.  For general L\'evy-driven CARMA processes, \cite{NoiseExtraction} show that the sequence $Z^\Delta$, with appropriate normalisation,  is a consistent estimator, as $\Delta\rightarrow 0$, of the increments of the driving L\'evy process, suggesting that the sequence $\{Z^\Delta_n\}_{n\in\bbz}$ is
approximately i.i.d. for small $\Delta$ even in the general case.

(ii) \, In the following corollaries and in Theorem~\ref{clt} we retain the assumptions on $Y$ and the sequence $\{m(n)\}$ made in the statement of Theorem \ref{consistency}.
\erem

\begin{cor}\label{4.2}
In the notation of Theorem~\ref{consistency}, the estimator,
\begin{equation}\label{estimator}
{\wh g}^\Delta(t):={{{\sqrt{{\hat v}_{m(n)}}}}\over{\sqrt{\Delta}}}{\hat\theta}_{m(n),r},\end{equation}
of $g(t)$ has error,
\begin{equation}\label{error:estimator}{\wh g}^\Delta(t)-g(t)=g^\Delta(t)-g(t)+\epsilon_n(t),\end{equation}
where $\epsilon_n(t)$ is asymptotically normal as $n\rightarrow\infty$, with asymptotic mean and variance, 0 and
$a_\Delta\sigma_\Delta^2/(n\Delta)$, respectively.
\end{cor}

\begin{proof}
Using \eqref{kernel:estimation:formula}, \eqref{innovations:algo:cons} and \eqref{estimator} we deduce from \eqref{innovations:algo:as} that, as $n\rightarrow\infty$,
\begin{equation}\label{4.5} {{\sqrt{n\Delta}}\over{\sigma_\Delta}}\epsilon_n(t)={{\sqrt{n\Delta}}\over{\sigma_\Delta}}({\wh g}^\Delta(t)-g^\Delta(t)) \, \Rightarrow \, {\rm N}(0,a_\Delta),\end{equation}
which is equivalent to the statement of the corollary.
\end{proof}

\bexam[The CARMA(1,0) process]\label{excar1}
Application of Corollary 4.2 to the $\CARMA(1,0)$ process using the results of Example \ref{Carma10ex} (with $\sigma^2=1$) immediately yields the representation
\begin{equation}\label{car1}
{\wh g}^\Delta(t)-g(t)=\epsilon_n(t)+\left({\sqrt{{{e^{2\lambda\Delta}-1}\over{2\lambda\Delta}}}}-1\right)e^{\lambda t},\end{equation}
where, as $n\rightarrow\infty$, $\epsilon_n(t)$ is asymptotically normal with mean $0$ and variance $(e^{2\lambda t}-1)/(2\lambda n\Delta)$.
The last term in \eqref{car1} tends to zero
as $\Delta\rightarrow 0$ and $\epsilon_n(t)$ converges in probability to $0$ if we allow $\Delta$ to depend on $n$ in such a way that $\Delta(n)\to0$ and $n\Delta(n)\rightarrow \infty$ as $n\rightarrow \infty$.
\eexam

In the following we shall suppose, as in Example \ref{excar1}, that $\Delta$ depends on $n$ in such a way that $\Delta(n)\rightarrow 0$ and $n\Delta(n)\rightarrow\infty$ as $n\rightarrow\infty$ and study the asymptotic behaviour of ${\wh g}^\Delta(t)$ as $n\rightarrow\infty$.

\begin{cor}\label{cor:const}
If $\Delta(n)\rightarrow 0$ and $n\rightarrow\infty$ in such a way that $n\Delta(n)\rightarrow\infty$ (i.e. such that the time interval over which the observations are made goes to $\infty$) then ${\wh g}^\Delta(t)$ is consistent for $g(t)$ for each fixed $t$.
\end{cor}

\begin{proof}
Under the conditions stated, the random variables $\epsilon_n(t)$ in \eqref{error:estimator} converge in probability to zero by Corollary~\ref{4.2} and the fact that $a_\Delta\sigma_\Delta^2\le {\rm Var}(Y(t))$.
The deterministic component of  \eqref{error:estimator}, $g^\Delta(t)-g(t)$, converges to zero by Theorem \ref{Woldapprox}.
\end{proof}

If we impose an additional condition on the rate at which $\Delta(n)$ converges to zero we obtain the following central limit theorem for our estimator ${\wh g}^\Delta(t)$. Its proof is given in the Appendix.

\begin{thm}\label{clt}
Suppose that $Y$ is a Gaussian $\CARMA(p,q)$ process or a general L\'evy-driven CARMA(1,0) process observed at times $k\Delta(n)$, $k=1,\ldots n$. If  $\Delta(n)\rightarrow 0$, $n\Delta(n)\rightarrow\infty$ and $n(\Delta(n))^3\rightarrow 0$ as $n \rightarrow \infty$, then, for fixed $t\ge 0$,
$$\sqrt{n\Delta}(\wh g^\Delta(t)-g(t)) \, \Rightarrow \, {\rm N}(0,\int_0^t g^2(u) du) ~{\rm as}~n\rightarrow\infty.$$
\end{thm}

\brem{We shall refer to the estimator (\ref{estimator}) as the {\it innovations estimator} of $g(t)$.   Instead of using the innovations estimates of $\psi^\Delta_{\lfloor t/\Delta\rfloor}$
and $\sigma^\Delta$ as in   (\ref{estimator}), we could also use  the coefficients and white-noise standard deviation obtained by using the Durbin-Levinson algorithm to fit
a high-order causal $\AR$ process with white-noise variance $\tau^2$ to the observed values of $Y^\Delta_k, k=1,\ldots n$,  and numerically inverting the fitted autoregressive polynomial $\phi(z)=1-\phi_1z-\ldots-\phi_pz^p$ to obtain the moving average representation
$$Y_n^\Delta=\sum_{j=0}^\infty \beta_j Z_{n-j}, ~~\{Z_{n}\}_{n\in\bbz}\sim{\rm WN}(0,\tau^2).$$
 where $\beta(z):=\sum_{j=0}^\infty\beta_jz^j=1/\phi(z), ~|z|\le 1$.  Substituting the estimators $\tau^2$ for $v_m$ and $\beta_r$ for $\theta_{m(n),r}$ gives the {\it Durbin-Levinson estimator} (of order $p$) for $g^\Delta(t)$.  Both of these estimators will be used in the examples which follow.   In practice it has been found that the Durbin-Levinson algorithm gives better results except when the fitted autoregressive polynomial has zeroes  very close to the unit circle.}
\erem

\bexam[Simulation results] {We now illustrate the performance of the estimators by applying them to realizations of the Gaussian CMA process $Y$ defined by (\ref{CMA}) with gamma kernel function,
\begin{equation}g(t)=t^{\nu-1}e^{-\lambda t}{\bf 1}_{(0,\infty)}(t), ~~\lambda>0,~ \nu>1/2,\end{equation}
with standard Brownian motion as the driving L\'evy process.  The variance of $Y_t$ is
$$\gamma_Y(0)=(2\lambda)^{1-2\nu}\Gamma(2\nu-1)$$
and the autocorrelation function is
$$\rho_Y(h)={{2^{3/2-\nu}}\over{\Gamma(\nu-1/2)}}|\lambda h|^{\nu-1/2}K_{\nu-1/2}(|\lambda h|),~~h\in\mathbb{R},$$
where the function $K_{\nu-1/2}$ is the modified  Bessel function of the second kind with index $\nu-1/2$ (\cite{abramowitz}, Section 9.6).
This is known as the Whittle-Mat\'{e}rn autocorrelation function (see \cite{whittle}) with parameter $\nu-1/2$, evaluated at $\lambda h$.

The simulations were carried out with $\lambda=1$ and two values of $\nu$, namely $\nu=1.05$ and $\nu=2$.  The kernel with $\nu=2$ is actually the kernel of the
$\CARMA(2,0)$ process \eqref{1.1} with $a(z)=(z+\lambda)^2$, $b(z)=1$ and $\sigma^2=1$.  The gamma kernel with $\nu=1.05$ however is the kernel of a CMA process but not of any CARMA process.

We first estimated ${\wh g}^\Delta$ by applying both the Durbin-Levinson and innovations algorithms to the {\it true} autocovariance functions which are known
for the simulated processes.  The purpose was to assess the effect of the sampling error when the {\it sample} autocovariances of the data are used.
The estimated kernel functions are shown in the upper rows of Figures 1-4.

The continuous-time sample-paths of $Y$ were simulated at the very finely-spaced times $k\Delta$ with $\Delta=10^{-6}$.
The sequences $Y^\Delta$ used to estimate $g$ were then sampled from these values using two different spacings, $\Delta=0.25$ and $\Delta=0.0625$
We then estimated the kernel function $g(\cdot)$ up to time $T=8$, and plotted $\wh{g}^\Delta((j+\frac{1}{2})\Delta)$ for $j=0,\ldots,N=32$ and for $j=0,\ldots,N=128$, respectively.  In the case of the innovations algorithm, we used (for the true as well as for the estimated autocovariances) values of the discrete autocovariance functions up to $3N$,
i.e. we chose $m$ in (\ref{estimator}) to be $3N$. {We could equally well have plotted ${\wh g}^\Delta((j+h)\Delta)$ for any $h\in[0,1)$, where the bias depends
on $h$.  However the variation becomes negligible as $\Delta\rightarrow 0$.  Some partial results regarding the optimal choice of $h$ are given in \cite{NoiseExtraction}.
}

\begin{figure}[ht]
\begin{center}
\includegraphics[width=1.0\textwidth]{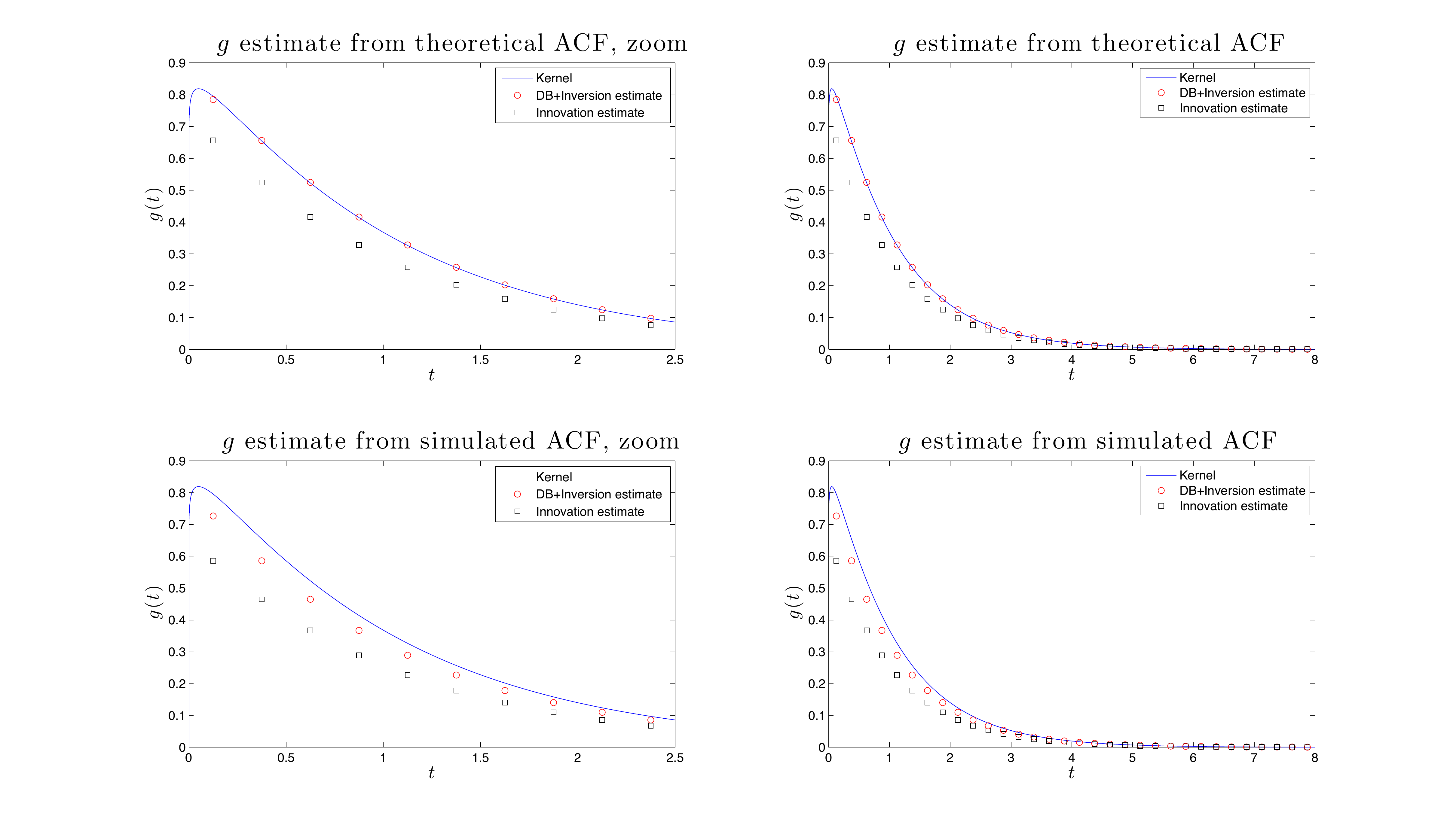}
\caption{Estimation of the gamma kernel for $\nu=1.05$ and $\Delta=2^{-2}$.}
\end{center}
\end{figure}

The results are shown in the bottom rows of Figures 1-4, where the squares denote the estimates from the innovations algorithm, and the circles denote those from the Durbin-Levinson algorithm. For reference the true kernel function is  plotted with a solid line.
Comparing the top and bottom rows of Figures 1-4 we find for the estimated autocovariance function an intrinsic finite-sample error, which influences the kernel estimation.
We notice that in all cases considered, the Durbin-Levinson algorithm gives better estimates. Furthermore, as expected,  the estimates for both algorithms improve with decreasing grid spacing.
The Durbin-Levinson algorithm provides estimates which are in good agreement with the original kernel function even for the coarse grid with $\Delta=0.25$.
}
\eexam

\begin{figure}[h!]
\begin{center}
\includegraphics[width=1.0\textwidth]{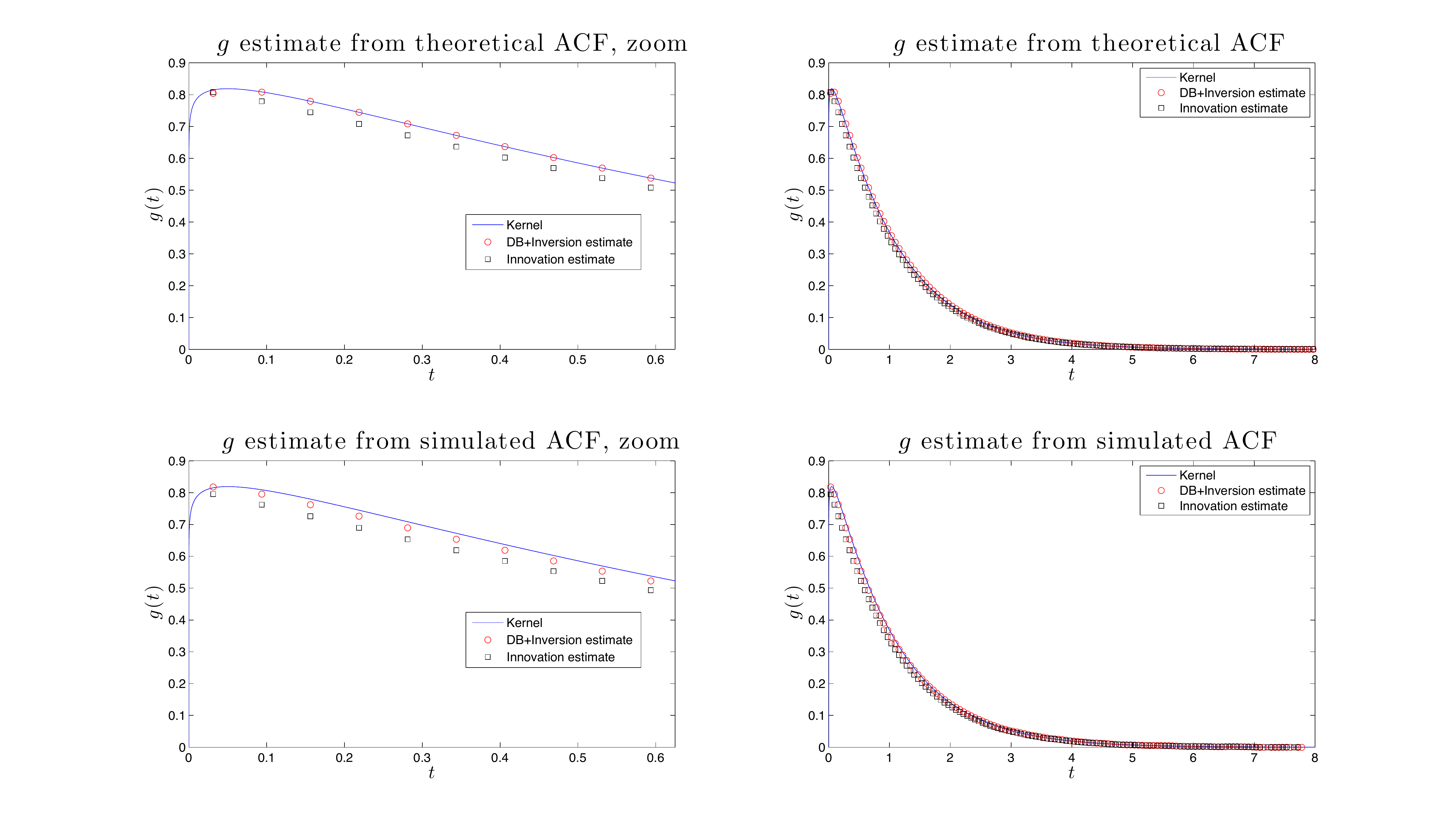}
\caption{Estimation of the gamma kernel   for $\nu=1.05$ and $\Delta=2^{-4}$.}
\end{center}
\end{figure}

\begin{figure}[h!]
\begin{center}
\includegraphics[width=1.0\textwidth]{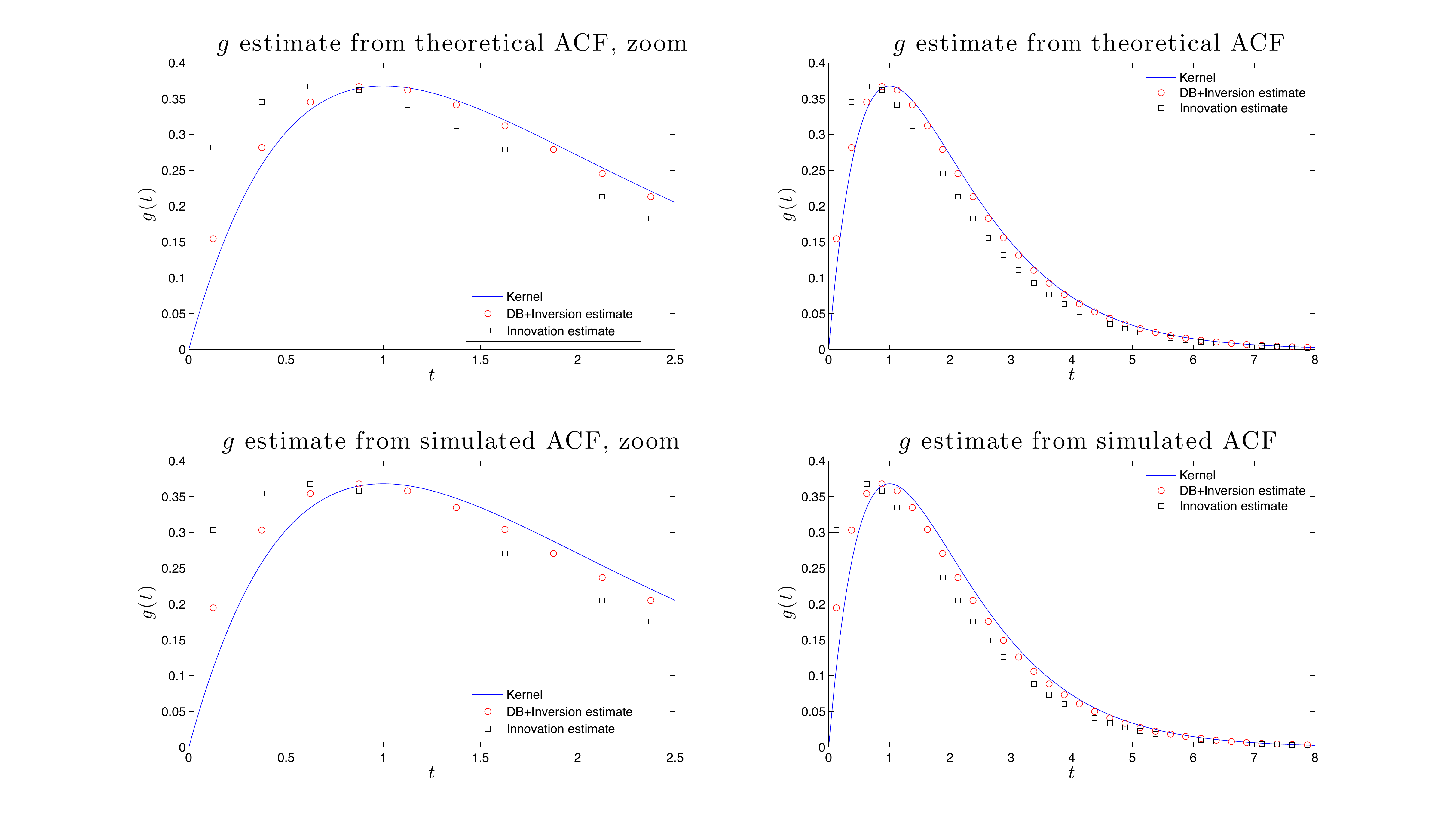}
\caption{Estimation of the gamma kernel   for $\nu=2$ (CAR(2) process) and $\Delta=2^{-2}$.}
\end{center}
\end{figure}

\begin{figure}[t!]
\begin{center}
\includegraphics[width=1.0\textwidth]{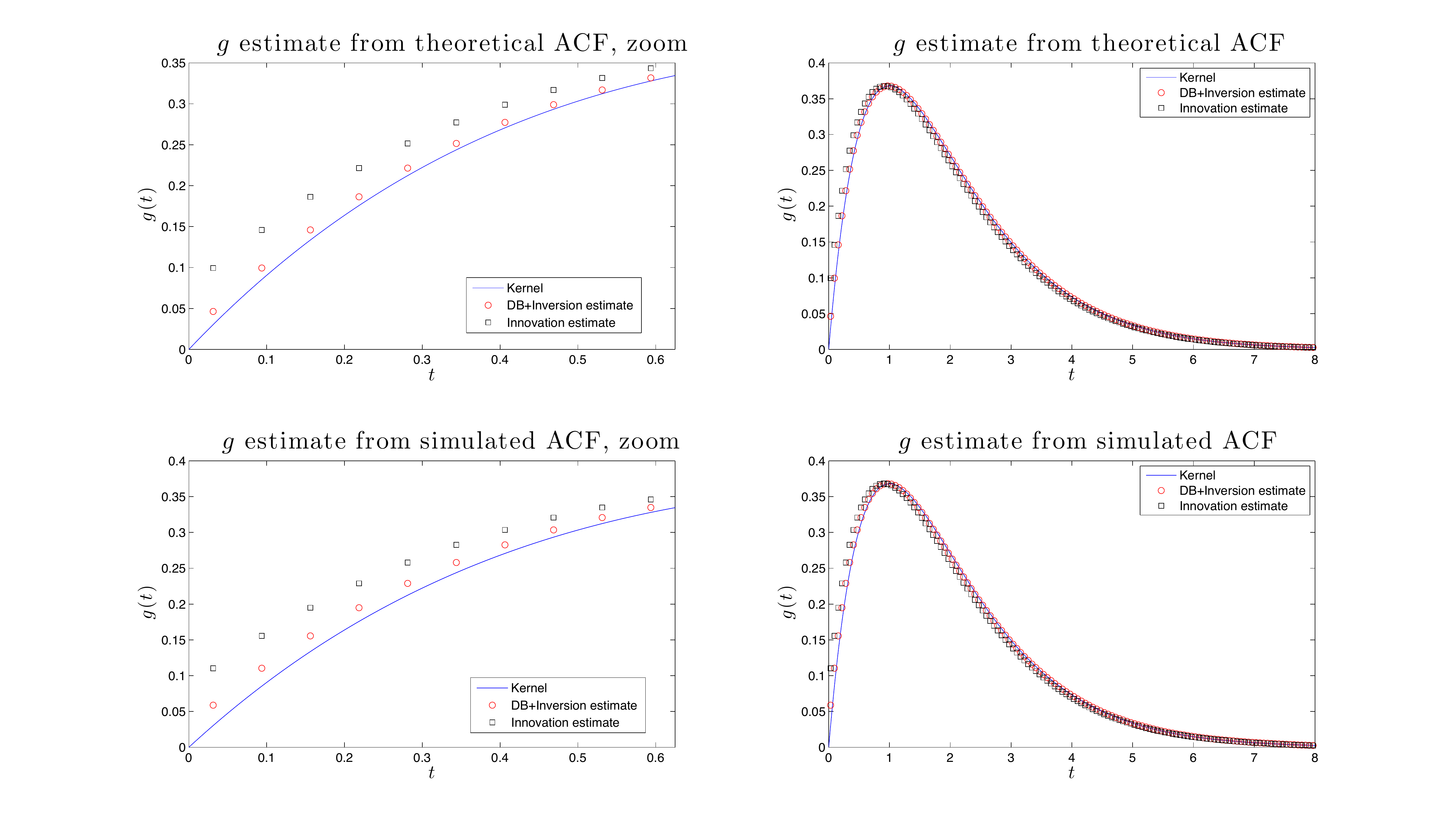}
\caption{Estimation of the gamma kernel   for $\nu=2$ (CAR(2) process) and $\Delta=2^{-4}$.}
\end{center}
\end{figure}

\section{An application to real data: mean flow turbulent velocities}\label{Brookhaven}
%
We now apply the Durbin-Levinson  algorithm of Section~\ref{sec:approx} to the Brookhaven turbulent wind-speed data, which consists of $20\times10^6$ measurements taken at 5000Hz (i.e. 5000 data points per second).
The series thus covers a total time interval of approximately 67 minutes and the sampling interval $\Delta$ is $2\times 10^{-4}$ seconds.
This dataset displays a rather high Reynolds number (about 17000), typical of turbulent phenomena.
A more detailed presentation of turbulence phenomena and an application of the  CMA model \eqref{CMA} in the context of turbulence modelling is given in \citet{fk:2011:1}; moreover we refer to \cite{Dhruva, techEnzo} for a precise description of the data, and to \cite{pope,frisch} for a comprehensive review of turbulence theory.
A CMA model \eqref{CMA}  with a gamma kernel as in Example~\ref{ga:ker} has been suggested as a parametric model in \cite{intermit2}.

Figure \ref{results:brook} a) shows the sample autocorrelation function up to 120 seconds, which appears to be exponentially decreasing. In general, the data are not significantly correlated after a lag of 100 seconds.

The estimated spectral density ${\wh f}_Y$ of $Y^\Delta$ is shown in Figure \ref{results:brook} b), plotted against the frequency $\varphi$, measured in cycles per second (Hz). The estimates marked by circles were estimated by Welch's method (\cite{welchmethod}) with segments of $2^{22}$ data points (circa 14 minutes), windowed with a Hamming window and  using an overlapping factor of 50\%. This method allows a significant reduction of the variance of the estimate, sacrificing some resolution in frequency. In order to have a better resolution near  frequency zero, we estimated the spectral density for $\varphi\leq 10^{-3}$ Hz with the raw periodogram (\cite{BD}, p. 322), which provides a better resolution in frequency at cost of a larger variance. The results are plotted in the leftmost part of Figure \ref{results:brook} b) with diamonds, and the two ranges of estimation are indicated by a vertical solid line. The spectral density is plotted on a log-log scale, so that any power-law relationship will be reflected by linearity of the graph.  The spectral density in the neighborhood of zero appears to be essentially constant, as is compatible with an exponentially decreasing autocorrelation function (such as the gamma kernel function of Example~\ref{ga:ker}).

For frequencies $\varphi$ between $10^{-2}$ and 200Hz,  $\log{\wh f}_Y$ decreases linearly with $\log \varphi$ with a slope of approximately $-5/3$, in accordance with Kolmogorov's $5/3$-law.   For comparison, the solid line corresponds to a spectral density proportional to $\varphi^{-5/3}$.  For $\varphi$ larger than 200Hz, the spectral density deviates from the $5/3$-law, decaying with a steeper slope. We note that a spectral density decaying as prescribed by  Kolmogorov's law in the neighborhood of $\infty$ would require a kernel behaving like $t^{-1/6}$ near to the origin, according to Proposition \ref{abel} (see below).

The estimated kernel function $\wh g^\Delta(t)$ is plotted in Figure~\ref{results:brook} c) on a log-linear scale in order to highlight the behaviour of the kernel estimate at both very large and very small values of $t$.
The estimated $g(t)$ decays rapidly with $t$, with small oscillations around zero for $t>100$ seconds.
As $t$ decreases from this value to roughly $10^{-3}$ seconds,  the estimated kernel increases in accordance with Kolmogorov's $5/3$-law, dropping off to zero as $t$ decreases further,  matching  the steeper decay of the spectral density at high frequencies evident in Figures 5 b) and 5 d).

Figure 6 d) shows the spectral density computed directly from the estimated kernel function ${\wh g}^\Delta$.  Its close
resemblance to the spectral density calculated by Welch's method provides justification for our estimator of $g$
even when there is no underlying parametric model.

\begin{figure}[t!]
\begin{center}
\includegraphics[width=1\textwidth]{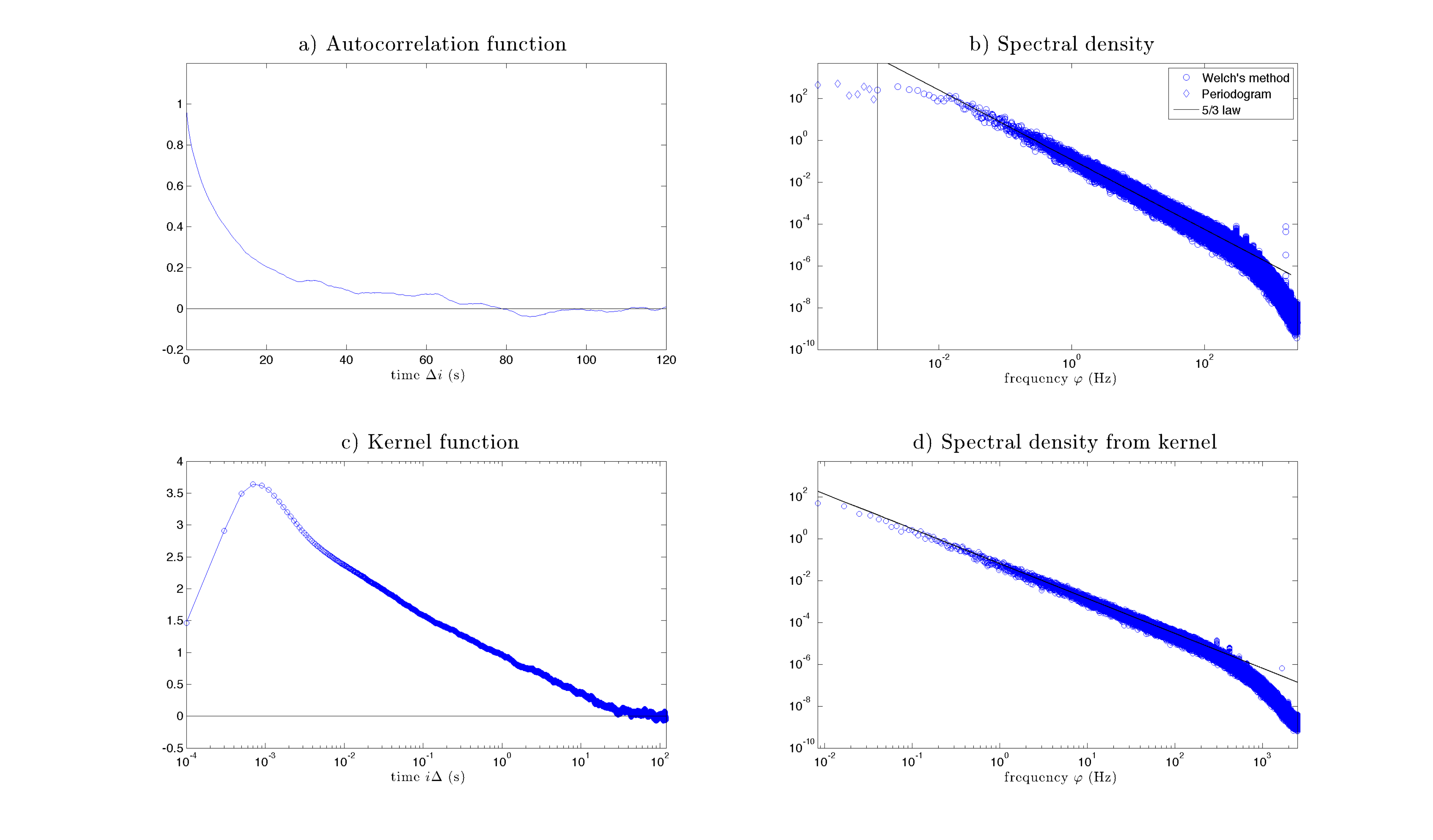}
\caption{Estimates for the Brookhaven dataset: a) autocorrelation function b) spectral density (Welch estimator and periodogram) c) kernel function (linear-log scale) d) spectral density computed using the estimated kernel.}
\label{results:brook}
\end{center}
\end{figure}

\section{Asymptotics for a class of sampled CMA processes as $\Delta\rightarrow 0$}\label{sec:asymp}

 {\cite{bfk:2011:1} derived first-order asymptotic expressions, as $\Delta\rightarrow 0$, for the spectral density $f_\Delta(\omega)$ of $Y^\Delta$, where $\omega$ denotes frequency in radians per unit time and $Y$ is  a $\CARMA(p,q)$ process
with $p-q\le 3$.
Although, as pointed out in Section 2, these asymptotic expressions are not sufficiently precise to establish the convergence of $g^\Delta$ to $g$,
they do reveal the {\it local second-order behaviour} of the process $Y$.  For example, if $Y$ is a CARMA$(p,p-1)$ process driven by a L\'evy process $L$ with Var$(L_1)=\sigma^2$,  then equations (15) and (19) of  \cite{bfk:2011:1} give, as $\Delta\rightarrow 0$,
$$f_\Delta(\omega)\sim{{\sigma^2\Delta}\over{4\pi(1-\cos\omega)}},\quad -\pi\le\omega\le\pi,$$
showing that the spectral density of the {\it normalized differenced} sequence $\{(Y_{n\Delta}-Y_{(n-1)\Delta})/{\sqrt{\Delta}}\}_{n\in\bbz}$ converges to that of white noise with variance $\sigma^2$  as $\Delta\rightarrow 0$.  In other words, for any fixed positive integer $k$, the sequence of observations $Y_{n\Delta}/\sqrt{\Delta}, ~n=1,\ldots,k$, from a second-order point of view, behaves as $\Delta\rightarrow 0$ like a sequence of observations of integrated white noise with white-noise variance $\sigma^2$.

In this section we derive analogous asymptotic approximations for the spectral densities of more general CMA processes and the implications for their local second-order behaviour.  Since we allow in this section for spectral densities with a singularity at zero we introduce the modified spectral domains,
$$\Omega_d:=[-\pi,\pi]\backslash\{0\} ~{\rm and}~\Omega_c:=(-\infty,\infty)\backslash\{0\}.$$We require the CMA processes to have spectral density satisfying a weak regularity condition at infinity.  To formulate this condition we first need a definition.}

\begin{defn}[Regularly varying function (cf. \cite{BGT}]
Let $f$ be a positive, measurable function defined on $(0,\infty)$.
If there exists $\rho\in\bbr$ such that
$$\lim_{x\rightarrow\infty}\frac{f(\lambda x)}{f(x)}=\lambda^{\rho}, \quad{\rm for \ all}\  \lambda>0,$$ holds,  $f$ is called a {\em regularly varying function of index $\rho$ at $\infty$}.
The convergence is then automatically locally uniform in $\la$.
We shall denote this class of functions by $\mathcal{R}_{\rho}(\infty)$. Furthermore we shall say that $f(\cdot)\in\mathcal{R}_\rho(0+)$ if and only if $f(1/\cdot)\in\mathcal{R}_{-\rho}(\infty)$.
\end{defn}

The characterization theorem for regularly varying functions (Theorem~1.4.1. in \cite{BGT}) tells us that $f\in\mathcal{R}_\rho(\infty)$ if and only if $f(x)=x^\rho \ell(x)$, where $\ell\in\mathcal{R}_0(\infty)$.

\begin{thm}\label{scaling:order}
Let $Y$ be the $\CMA$ process  \eqref{CMA} with strictly positive spectral density $f_Y$ such that
$f_Y\in\mathcal{R}_{-\alpha}(\infty)$, where $\alpha>1$, i.e., for $\ell\in\mathcal{R}_0(\infty)$,
\begin{equation}\label{reg:density}
f_Y(\omega)=|\omega|^{-\alpha}\ell(|\omega|),\quad \omega\in\Omega_c.
\end{equation}
Then
the following assertions hold.\\[2mm]
(a) \,
The spectral density of the sampled process $Y^\Delta$ has for $\Delta\rightarrow 0$ the asymptotic representation
\begin{equation}\label{slowv:sample}
f_\Delta(\omega)\sim \ell(\Delta^{-1})\Delta^{\alpha-1}
\left[|\omega|^{-\alpha}+(2\pi)^{-\alpha}\zeta\left(\alpha,1-\frac{\omega}{2\pi}\right)+(2\pi)^{-\alpha}\zeta\left(\alpha,1+\frac{\omega}{2\pi}\right)\right],\quad \omega\in\Omega_d,
\end{equation}
 where $\zeta(s,r)$ is the Hurwitz zeta function, defined as
$$\zeta(s,r):=\sum_{k=0}^\infty\frac{1}{(r+k)^s},\quad \Re(s)>1,\quad  r\neq 0, -1, -2,\ldots.$$
(b)\,
The right hand side of \eqref{slowv:sample} is not integrable for any $\Delta>0$.
However, the corresponding asymptotic spectral density of the differenced sequence $(1-B)^{\alpha/2 }Y^\Delta$ is integrable for each fixed $\Delta>0$ and  the spectral density of
\begin{equation}\label{rescaled}
\frac{(1-B)^{\alpha/2}}{\ell(\Delta^{-1})^{1/2}\Delta^{(\alpha-1)/{2}}}{Y^\Delta}\end{equation}
converges as $\Delta\rightarrow 0$ to that of a short-memory
stationary process{, i.e. a stationary process with spectral density bounded in a neighbourhood of the origin}.\\[2mm]
(c)\, The variance of the innovations $\{Z^\Delta_n\}_{n\in\bbz}$ in the Wold representation (\ref{Wold1}) of $Y^\Delta$ satisfies
$$\sigma^2_\Delta\sim 2\pi C_\alpha \ell\left(\Delta^{-1}\right) \Delta^{\alpha-1}, \quad \Delta\rightarrow 0,$$
where
\begin{equation}\label{coefficients}
C_\alpha=\exp\left\{\frac{1}{2\pi}\int_{-\pi}^{\pi}\log\left[|\omega|^{-\alpha}+(2\pi)^{-\alpha}\zeta\left(\alpha,1-\frac{\omega}{2\pi}\right)+(2\pi)^{-\alpha}\zeta\left(\alpha,1+\frac{\omega}{2\pi}\right)\right]d\omega\right\}.
\end{equation}
\end{thm}

\brem{\rm (i) Theorem~\ref{scaling:order}(b) means that, from a second-order point of view, a sample $\{Y^\Delta_n, n=1,\ldots,k\}$ with $k$ fixed and
$\Delta$ small resembles a sample from an $({\alpha/2})$-times
integrated short-memory stationary sequence. If in (b) we replace
$(1-B)^{\alpha/2}$ by $(1-B)^\gamma$ where $\gamma>(\alpha-1)/2$,
then the conclusion holds for the overdifferenced process.
 If, for example, we difference at order
$\gamma=\lfloor(\alpha+1)/2\rfloor$ (the
smallest integer greater than $(\alpha-1)/2$) we get a stationary process.
{In particular, if $1<\alpha<3$, then $\lfloor(\alpha+1)/2\rfloor=1$ and, by (\ref{slowv:sample}) and (\ref{filt:process}), the differenced sequence $(1-B)Y^\Delta$ has the asymptotic spectral density, as $\Delta\rightarrow 0$,
$$\ell(\Delta^{-1})\Delta^{\alpha-1} 2(1-\cos\omega) \left[|\omega|^{-\alpha}+(2\pi)^{-\alpha}\zeta\left(\alpha,1-\frac{\omega}{2\pi}\right)+(2\pi)^{-\alpha}\zeta\left(\alpha,1+\frac{\omega}{2\pi}\right)\right], \, \omega\in\Omega_d.$$This is the spectral density of the increment process of a self-similar process with self-similarity parameter $H=(\alpha-1)/2$ (see \cite{StatMethodLong}, eq. (2)).  In general, for $\alpha>1$ the asymptotic autocorrelation function of the filtered sequence has unbounded support. The only notable exception is when $\alpha$ is even, where the asymptotic autocorrelation sequence is the one of a moving-average process with order $\alpha/2$, as in \cite{bfk:2011:1} or in Example \ref{oss:carma}.} \\[2mm]
(ii) \, The constant $C_\alpha$ of (\ref{coefficients}) is shown as a function of $\alpha$ in Figure~6.   The values, when $\alpha$ is an even positive integer, can be derived
from (\ref{sigmasquared}) since CARMA processes constitute a subclass of the processes covered by the theorem (see Example \ref{oss:carma}).  It is clear from (\ref{coefficients}) that $C_\alpha$  is exponentially bounded as $\alpha\rightarrow\infty$.}
\erem
\begin{figure}[h!]
\begin{center}
\includegraphics[width=0.8\textwidth, height=0.4\textwidth]{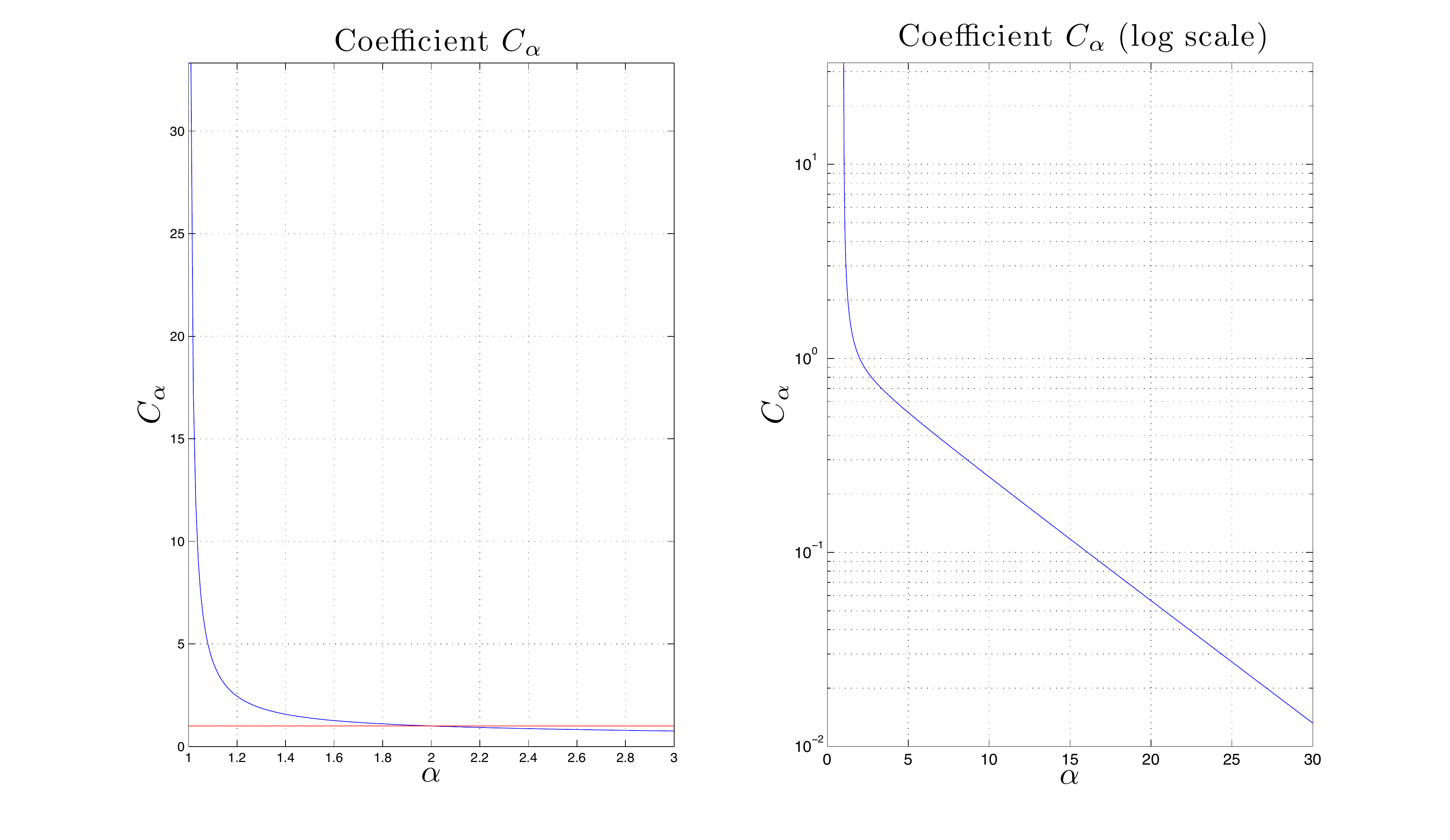}
\caption{The constant $C_\alpha$, as a function of the index of regular variation $\alpha$, is shown on the left using a linear scale and  on the right using a logarithmic scale. From Corollary 3.4 (a) of \cite{bfk:2011:1} we know that $C_2=1$. The horizontal line indicates the value 1.}
\end{center}
\end{figure}
{
\begin{cor}\label{corol:incr}
Let $Y$ be a CMA process satisfying the assumptions of Theorem~\ref{scaling:order} with $1<\alpha<2p+1$. Then as $\Delta\rightarrow 0$,
$$\bbe[((1-B)^pY^\Delta_{n})^2]\sim 2^p S_{p,\alpha}\ell(\Delta^{-1})\Delta^{\alpha-1},$$
where
$$S_{p,\alpha}=\int_{-\pi}^{\pi}(1-\cos\omega)^p\left[|\omega|^{-\alpha}+(2\pi)^{-\alpha}\zeta\left(\alpha,1-\frac{\omega}{2\pi}\right)+
(2\pi)^{-\alpha}\zeta\left(\alpha,1+\frac{\omega}{2\pi}\right)\right]d\omega.$$
\end{cor}

\bproof
By stationarity we have $\bbe[(1-B)^pY^\Delta_{n}]=0$ and, hence $\bbe[((1-B)^pY^\Delta_{n})^2]$ is the variance, of $((1-B)^pY^\Delta_{n})$ which can be calculated as the integral of its spectral density.  Thus
$$\bbe[((1-B)^pY^\Delta_{n})^2]= 2^p \int_{-\pi}^\pi (1-\cos\omega)^p f_\Delta(\omega)d\omega.$$
Using the inequalities (\ref{bounds}) and Lebesgue's dominated convergence theorem, we find that
as $\Delta\rightarrow 0$,
$${{1}\over{\ell(\Delta^{-1})\Delta^{\alpha-1}}}\int_{-\pi}^\pi (1-\cos \omega)^p f_\Delta(\omega)d\omega\rightarrow\int_{-\pi}^\pi (1-\cos\omega)^p\sum_{k=-\infty}^\infty |2k\pi+\omega|^{-\alpha}d\omega,$$
which, with the previous equation and (\ref{sum}), gives the result.
\eproof}

{The kernel of the CMA process (\ref{CMA}) and its spectral density are  linked by the formula,
\begin{equation}\label{specdens}f_Y(\omega)={{1}\over{2\pi}}\int_{-\infty}^\infty e^{-i\omega h}\gamma(h) dh ={{\sigma^2}\over{2\pi}}|\tilde g(\omega)|^2,~\omega\in\Omega_c,\end{equation}where
$${\tilde g}(\omega):=\int_{-\infty}^\infty e^{i\omega t}g(t)dt,$$
 Moreover, it has long been known that local properties of a function imply global properties of its Fourier transform (see e.g. \cite{Titchmarsh}, Theorems 85 and 86).

An Abelian theorem of \cite{AbelFourier} allows us to show, under the conditions of the following proposition,  that CMA processes with regularly varying kernels at the origin have regularly varying spectral densities at infinity.

\begin{prop}\label{abel}
Let $Y$ be a CMA process with kernel $g\in\mathcal{R}_{\nu-1}(0+)$ for $\nu>1/2$.
Assume that the derivatives in 0 satisfy the assumptions\\[2mm]
(\textbf{A1}) \, $g^{(\lfloor\nu\rfloor)}(0+)\neq0$; \\
(\textbf{A2}) \, $g^{(\lfloor\nu-1\rfloor)}\in\mathcal{R}_{\alpha}(0+)$ for $\alpha\in[0,1)$ (with $g^{(-1)}:=\int_0^tg(s)ds$);\\
(\textbf{A3}) \,  For some $x_0>0$,
$$q(u):=\sup_{x\leq x_0}\sup_{0\leq w\leq v\leq 1}\left|\frac{g^{(\lfloor\nu-1\rfloor)}((u+v+w)x)-g^{(\lfloor\nu-1\rfloor)}((u+v)x)-g^{(\lfloor\nu-1\rfloor)}((u+w)x)+g^{(\lfloor\nu-1\rfloor)}(u x) }{g^{(\lfloor\nu-1\rfloor)}(x)}\right|,$$
is bounded and integrable on $[1,\infty).$\\[2mm]
Then
$$ f_Y(|\cdot|)\in\mathcal{R}_{-2\nu}(\infty).$$
\end{prop}

\begin{proof}
Under conditions \emph{(\textbf{A1})}-\emph{(\textbf{A3})} we can apply Theorem~2 of \cite{AbelFourier}, which yields
\begin{equation}\label{ABELFourier}{\tilde g}(|\omega|)\sim  \Gamma(\nu+1)e^{\pm i\nu\pi/2}\int_0^{1/|\omega|}g(s)ds,\quad \omega\rightarrow \pm\infty.
\end{equation}
Moreover, Karamata's  theorem (Theorem~1.5.11(ii) in \cite{BGT})
gives
$$\int_{0}^{1/|\omega|}g(s)ds=\int_{|\omega|}^\infty{s^{-2}g(1/s)ds}\sim |\omega|^{-1} g(1/|\omega|)/\nu, \quad \omega\rightarrow\pm\infty$$
where we used the fact that $g(\cdot)\in\mathcal{R}_{\nu-1}(0+)$ means $g(1/\cdot)\in\mathcal{R}_{-\nu+1}(\infty)$.

Substituting \eqref{ABELFourier} into \eqref{specdens} and recalling that $\Gamma(\nu+1)=\nu\Gamma(\nu)$, we obtain
$$ f_Y(|\omega|) =\frac{1}{2\pi} |F(g)|^2(\omega)\sim\frac{\Gamma^2(\nu)}{2\pi} |\omega|^{-2}g^2(1/|\omega|), \quad \omega\rightarrow\pm\infty,$$
which gives the desired result.
\end{proof}

\brem \rm
Condition (\emph{\textbf{A2}}) can be replaced by a monotonicity condition on the derivative $g^{(\lfloor\nu\rfloor)}(\cdot)$ near the origin, so that the monotone density theorem (\cite{BGT}, Theorem~1.7.2.) can be applied.
\erem

\bexam\label{oss:carma} [The $\CARMA(p,q)$ process]\\
The $\CARMA(p,q)$ process $Y$ defined by \eqref{1.1} has spectral density $f_Y(\omega)=\sigma^2|b(i\omega)|^2/(2\pi |a(i\omega)|^2)$, which clearly has the form
$$
f_Y(\omega)=|\omega|^{-\alpha}\ell(|\omega|), ~\omega\in\mathbb{R},
$$where $\alpha=2(p-q)$ and $\lim_{\omega\rightarrow\infty}\ell(|\omega|)=\sigma^2/(2\pi)$.
Hence, by Theorem~\ref{scaling:order}(c), the white noise variance
in the Wold representation of $Y^\Delta$ satisfies as $\Delta\rightarrow 0$,
\begin{equation}\label{CARMAWNV}
\sigma_{\Delta}^2\sim \sigma^2C_{2(p-q)}\Delta^{2(p-q)-1},
\end{equation}
where $C_{2(p-q)}$ can be calculated from (\ref{coefficients}).  However $C_{2(p-q)}$ can also be
calculated from (\ref{sigmasquared}) as $C_{2(p-q)}=[(2(p-q)-1)!\prod_{i=1}^{p-q-1}\eta(\xi_i)]^{-1}$,
where $\eta(\xi_i)$ was defined in (\ref{eta1}).
Theorem~\ref{scaling:order}(b) implies that  the spectral density of $\Delta^{q-p+1/2}(1-B)^{p-q}Y^\Delta$ converges to that of a short memory stationary
process.  From Theorem \ref{MArepresentation1} we get the more precise result that the spectral density of
$C^{1/2}_{2(p-q)}\Delta^{q-p+1/2}(1-B)^{p-q}\prod_{i=1}^q(1+\eta(\xi_i)B)^{-1}Y^\Delta$ converges to that of white noise with variance $\sigma^2$.
\eexam

\bexam\label{ex:ficarma}[The $\FICARMA(p,d,q)$ process]\\
The fractionally integrated causal \FICARMA$(p,d,q)$ process (\citet{BroMar05}) has spectral density
\begin{equation}\label{fi:spec}
f_Y(\omega)=\frac{\sigma^2}{2\pi}\frac{1}{|\omega|^{2d}}\left|\frac{b(i\omega)}{a(i\omega)}\right|^2,\quad \omega\in\Omega_c,
\end{equation}
with $a(\cdot)$ and $b(\cdot)$ as in  \eqref{1.1} and $0<d<0.5$.  Hence
$$\label{carma:spec}
f_Y(\omega)=|\omega|^{-\alpha}\ell(|\omega|),\quad \omega\in\Omega_c,
$$where $\alpha=2(p+d-q)$ and $\lim_{\omega\rightarrow\infty}\ell(|\omega|)=\sigma^2/(2\pi)$.  The spectral density (\ref{fi:spec})
has a singularity at frequency $0$ which gives rise to the slowly decaying autocorrelation function associated with long memory.
Applying Theorem \ref{scaling:order}(c) as in Example \ref{oss:carma}, the white noise variance
in the Wold representation of $Y^\Delta$ satisfies as $\Delta\rightarrow 0$
\begin{equation}\label {FICARMAWNV}
\sigma_{\Delta}^2\sim \sigma^2C_{2(p+d-q)}\Delta^{2(p+d-q)-1},
\end{equation}
where $C_{2(p+d-q)}$ can be calculated from (\ref{coefficients}).  As $\Delta\rightarrow 0$, the asymptotic spectral density $f_\Delta$ of $Y^\Delta$ is given by (\ref{slowv:sample}) with $\alpha=2(p+d-q)>1$ and is therefore not integrable for any $\Delta>0$.
However Theorem \ref{scaling:order}(b) implies that  the spectral density of $\Delta^{q-p-d+1/2}(1-B)^{p+d-q}Y^\Delta$ converges to that of a short memory stationary process.
\eexam

Our next two examples are widely used in the modelling of turbulence.  Kolmogorov's famous 5/3 law (see \cite{frisch} Section 6.3.1, \cite{pope} Section 6.1.3) suggests a regularly varying spectral density model for turbulent flows.

\bexam\label{kaimal}[Two turbulence models]\\
Denote by $\ov U$ the mean flow velocity, with $\ell$ the integral scale parameter and define $\ov\ell=\ell/\ov U$.\\
(i) \, The \cite{karman} spectrum models the isotropic energy spectrum.
Its spectral density is, for $C, c_\ell\in\bbr$, given by
$$
f_Y(\omega)=C{\ov U}^{-2/3}|\omega|^{-5/3}\left(\frac{\omega^2}{{\omega^2+{c_\ell}/{{\ov\ell}^2}}}\right)^{17/6},\quad\omega \in \Omega_c.
$$
Moreover, $f_Y\in\calr_{-5/3}$, so it has a representation \eqref{reg:density} and the conclusions of Theorem~\ref{scaling:order} hold with $\al=5/3$.
 \\[2mm]
(ii) \, The Kaimal spectrum for the longitudinal component of the energy spectrum is the current standard of the International Electrotechnical Commission; cf.~\cite{IEC}.
The spectral density is given by
\begin{equation}
f_Y(\omega)=v\frac{4\ov\ell}{(1+6\ov\ell\omega)^{5/3}},\quad \omega\in\Omega_c,
\end{equation}
where $v$ is the variance of $Y$.
Moreover, $f_Y\in\calr_{-5/3}$, so it has a representation \eqref{reg:density} and the conclusions of Theorem~\ref{scaling:order} hold with $\al=5/3$.
\eexam

\bexam\label{ga:ker} [The gamma kernel]\\
{The gamma kernel $g$ defined in \eqref{innovations:algo:cons} belongs to $\calr_{\nu-1}(0+)$ and satisfies the assumptions of Proposition~\ref{abel}.}
Its Fourier transform is ${\tilde g}(\omega)=\Gamma(\nu) (\lambda -i \omega)^{-\nu }$.  If $Y$ has the kernel $g$, then from  (\ref{specdens}) its spectral density is
$$f_Y(\omega)=\frac{\sigma^2}{2\pi}|{\tilde g}(\omega)|^2=\frac{\sigma^2}{2\pi}\frac{\Gamma^2(\nu)}{ (\lambda^2 + \omega^2)^{\nu }}
=\omega^{-2\nu}\frac{\sigma^2\Gamma^2(\nu)}{ 2\pi \left((\lambda/\omega)^2 +1\right)^{\nu }},\quad\omega\in\Omega_c.$$
which belongs to $\calr_{-2\nu}(\infty)$ with slowly varying function $\ell$ such that $\lim_{\omega\to\infty}\ell(\omega)=\sigma^2 \Gamma^2(\nu)/2\pi$.\\[2mm]
Note that if $\nu=5/6$, then $f_Y$, like the von K\'arm\'an spectral density of Example \ref{kaimal} (i), decays as $\omega^{-5/3}$ for $\omega\rightarrow\infty$, in accordance with Kolmogorov's $5/3$ law.\\[2mm]
Theorem \ref{scaling:order} gives the asymptotic form of the spectral density of the sequence $\{(1-B)^\nu Y_n^\Delta\}_{n\in\bbz}$ as $\Delta\rightarrow 0$,
\beao
h^\Delta(\omega) &\sim& \sigma^2\Gamma^2(\nu)(2\pi)^{-1}2^{\nu}\Delta^{2\nu-1}(1-\cos\omega)^{\nu}~\times\\
&&\left[|\omega|^{-2\nu}+(2\pi)^{-2\nu}\zeta\left(2\nu,1-\frac{\omega}{2\pi}\right)+(2\pi)^{-2\nu}\zeta\left(2\nu,1+\frac{\omega}{2\pi}\right)\right],\quad\omega\in\Omega_d.
\eeao
{The second-order structure function, $S_2(\Delta):=\bbe[(Y_\Delta-Y_0)^2]= 2 \gamma_Y(0) (1-\rho_Y(\Delta))$, plays an important role in the physics of turbulence. For the gamma kernel, $\gamma_Y(0)$ and $\rho_Y(h)$ were specified in Example \ref{ga:ker}.
Using those expressions and  the asymptotic behaviour as $\Delta\rightarrow 0$ of ${K}_{\nu-1/2}(\Delta)$ (see \cite{abramowitz}, Section 9.6), we obtain the asymptotic formulae,
\begin{align*}
{{S_2(\Delta)}\over{2\gamma_Y(0)}}=
\left\{\begin{array}{lr}
2^{1-2\nu}\displayfrac{\Gamma(3/2-\nu)}{\Gamma(\nu+1/2)}(\lambda\Delta)^{2\nu-1}+O(\Delta^{2}),&1/2<\nu<3/2,\\
&\\
\frac{1}{2}(\lambda\Delta)^2|\log\Delta|+O(\Delta^{3}),&\nu=3/2,\\
&\\
\displayfrac{1}{4(\nu-3/2)}(\lambda\Delta)^2+O(\Delta^{2\nu-1}),&\nu>3/2,
\end{array} \right.
\end{align*}
which can be found in \cite{pope}, Appendix G, and \cite{limit3}.  The first of these formulae can also be obtained as a special
case of Corollary \ref{corol:incr} with $p=1$.
}
\eexam

\section{Conclusions}
We studied the behaviour of the sequence  of observations $Y^\Delta$ obtained when a CMA process of the form \eqref{CMA}
is observed on a grid with spacing $\Delta$ as $\Delta\rightarrow 0$.

In the particular case when $Y$ is a CARMA process we obtained a more refined asymptotic representation of
the sampled process than that found by Brockwell et al. (2012) and used it to show the pointwise convergence
as $\Delta\rightarrow 0$ of a sequence of functions defined in terms of the Wold representation of the sampled process
to the kernel $g$.  This suggested a non-parametric approach to the estimation of $g$ based on estimation of the
coefficients and white noise variance of the Wold representation of the sampled process.

For a larger class of CMA processes we found results analogous to those of Brockwell et al. (2012) and examined their
implications for the local second-order properties of such processes, which include in particular fractionally integrated CARMA processes.

Finally we applied the non-parametric procedure for estimating $g$ to simulated and real data with positive results.

\subsection*{Acknowledgment}
P.J.B. gratefully acknowledges the support of the NSF Grant DMS-1107031 and, together with C.K., financial support of the Institute for Advanced Studies of the Technische Universit\"at M\"unchen (TUM-IAS).
V.F. would like to thank Ole Barndorff-Nielsen and J\"urgen Schmiegel for interesting discussions during the very enjoyable period spent at Aarhus University during February 2010.
Furthermore, he would like to thank Richard Davis from Columbia University for the many fruitful discussions during R.D.'s visits as Hans Fischer Senior Fellow of the TUM-IAS. The work of V.F. was supported by the International Graduate School of Science and Engineering (IGSSE) of the Technische Universit\"at M\"unchen.  We are grateful to an Associate Editor and two referees for their valuable comments which led to substantial improvement of the presentation of the results.

\bibliographystyle{kluwer}
\bibliography{bib}

\bigskip

\section*{Appendix}

\noindent {\bf \large Proof of Theorem \ref{MArepresentation1}}\\
It follows from (\ref{sampledCARMAspec}) that the spectral density $f_\Delta(\omega)$ of the sampled CARMA process $Y^\Delta$ is $-\sigma^2/(2\pi)$ times the sum of the residues at the singularities of the integrand
in the left half-plane, or more simply $\sigma^2/(4\pi)$ times the residue of the integrand at $\infty$, which is much simpler to calculate.
Thus,
$$f_\Delta(\omega)={{\sigma^2}\over{4\pi}}{\rm Res}_{z=\infty}\left[{{b(z)b(-z)}\over{a(z)a(-z)}} {{\sinh(\Delta z)}\over{\cosh(\Delta z)-\cos(\omega)}}\right],\quad -\pi\le\omega\le\pi.$$
The spectral density can also be expressed as a power series,
\begin{equation}\label{PSsampledCARMA}
f_\Delta(\omega)=\frac{\si^2}{4\pi}\sum_{j=0}^\infty{{\sigma^2\Delta^{2j+1}}}r_jc_j(\omega),\quad -\pi\le\omega\le\pi,
\end{equation}
where
$c_k(\omega)$ is the coefficient of $z^{2k+1}$ in
$$\sum_{k=0}^\infty c_k(\omega) z^{2k+1}={{\sinh z}\over{\cosh z-\cos \omega}},$$and
$$r_j:= {\rm Res}_{z=\infty}\left[z^{2j+1}{{b(z)b(-z)}\over{a(z)a(-z)}}\right],$$
i.e. the coefficient of $z^{2j}$ in the power series expansion,
\begin{equation}\label{rj}
\sum_{j=0}^\infty r_jz^{2j}=(-z^2)^{p-q-1}{{\prod_{i=1}^q (1-\mu_i^2z^2)}\over{\prod_{i=1}^p (1-\lambda_i^2 z^2)}},\end{equation}where
 $a(z)=\prod_{i=1}^p(z-\lambda_i)$ and $b(z)=\prod_{i=1}^q(z-\mu_i)$.
Denoting by $f_{MA}$ the spectral density of the moving average, $X_n:=\theta^\Delta(B)Z^\Delta_n$, we find from (\ref{ARMA}) that
$$f_{MA}(\omega)=2^p e^{-a_1\Delta}f_\Delta(\omega)\prod_{j=1}^p (\cosh(\lambda_j\Delta)-\cos(\omega)), ~~-\pi\le\omega\le \pi$$
and hence, by (\ref{PSsampledCARMA}),
$$f_{MA}(\omega)={{2^p\sigma^2e^{-a_1\Delta}}\over{4\pi }}
\prod_{i=1}^p\Big(1-\cos\omega+\sum_{j=1}^\infty {{(\lambda_i\Delta)^{2j}}\over{(2j)!}}\Big)
\sum_{k=0}^\infty r_k c_k(\omega)\Delta^{2k+1},\quad -\pi\le\omega\le\pi.$$
This expression can be simplified by reexpressing it in terms of $x:=1-\cos\omega$.
Thus
\begin{equation}\label{PSfMA}
f_{MA}(\omega)={{2^p\sigma^2e^{-a_1\Delta}}\over{4\pi }}
\prod_{i=1}^p\Big(x+\sum_{j=1}^\infty {{(\lambda_i\Delta)^{2j}}\over{(2j)!}}\Big)
\sum_{k=0}^\infty r_k \alpha_k(x)\Delta^{2k+1},
\end{equation}
where $\alpha_k(x)$ is the coefficient of $z^{2k+1}$ in the expansion,
$$\sum_{k=0}^\infty \alpha_k(x) z^{2k+1}={{\sinh z}\over{\cosh z-1+x}}.$$
In particular $\alpha_0(x)=1/x, \alpha_1(x)=(x-3)/(3!x^2)$ and $\alpha_2(x)=(x^2-15x+30)/(5!x^3)$.
More generally, $\alpha_k(x)$ has the form.
\begin{equation}\label{alphaj}\alpha_k(x)={{1}\over{(2k+1)!x^{k+1}}}\prod_{i=1}^k(x-\xi_{k,i}),\end{equation}
where
\begin{equation}\label{product}
\prod_{i=1}^k\xi_{k,i}= (2k+1)!\,2^{-k},
\end{equation}
and the product, when $k=0$, is defined to be $1$.
Since $\alpha_{p-q-1}(x)$ plays a particularly important role in what follows, we shall relabel it as $\alpha(x)$ and denote its zeroes more simply as
$$\xi_i:=\xi_{p-q-1,i}, ~~i=1, \ldots,p-q-1.$$
From (\ref{PSfMA}), with the aid of (\ref{rj}) and (\ref{alphaj}), we can now derive an asymptotic approximation to and factorization of $f_{MA}(\omega)$.
Observe first that the expression on the right of (\ref{PSfMA}), in spite of its forbidding appearance, is in fact a polynomial in $x$ of degree less than $p$.
We therefore
collect together the coefficients of $x^{p-1}, x^{p-2}, \ldots, x^0$.  This gives (using the identity (\ref{product})) the asymptotic expression as $\Delta\rightarrow 0$,
\begin{equation}\label{PSfMA2}
f_{MA}(\omega)={{2^p\sigma^2e^{-a_1\Delta}\Delta^{2(p-q)-1}}\over{4\pi }}\Big[x^pr_{p-q-1}\alpha(x)+o(1)   +\sum_{j=1}^q\rho_jx^{q-j}\Delta^{2j}
\Big],\end{equation}with
$$\rho_j= (-2)^{-(p-q-1+j)}\Big[r_{p-q-1+j}-r_{p-q-2+j}\sum_{i=1}^p\lambda_i^2\Big]+o(1)$$
$$\hskip -1.1in=2^{-(p-q-1+j)}\sum\mu_{i_1}^2\ldots\mu_{i_j}^2 +o(1),$$where the second line follows from (\ref{rj}) and the sum on
the second line is over all subsets of size $j$ of the $q$ zeroes of the polynomial $b(z)$.
Replacing $r_{p-q-1}$ in (\ref{PSfMA2}) by $(-1)^{p-q-1}$, substituting for $\alpha(x)=\alpha_{p-q-1}(x)$ from (\ref{alphaj}) and using the continuity of the zeroes of a
polynomial as functions of its coefficients,  we can rewrite (\ref{PSfMA2}) (recalling that $x:=1-\cos\omega$ and $\xi_i, i=1,\ldots,p-q-1$ are the zeroes of $\alpha(x)$) as
\begin{equation}\label{finalPSfMA} f_{MA}(\omega)={{\Delta(-\Delta^2)^{p-q-1}2^p\sigma^2e^{-a_1\Delta}}\over{[2(p-q)-1]!4\pi}} \prod_{i=1}^{p-1-q}\left[x-\xi_i(1+o(1))\right] \prod_{k=1}^q\left[x+{{\mu_k^2\Delta^2}\over{2}}(1+o(1))\right].\end{equation}
To complete the factorization of $f_{MA}(\omega)$, observe that we can write
\begin{equation}\label{factor1}x+{{\mu_k^2\Delta^2}\over{2}}(1+o(1))={{1}\over{2\zeta_k}}[1-(\zeta_k+o(\Delta))e^{-i\omega}][1-(\zeta_k+o(\Delta)e^{i\omega}],\end{equation}
where
\begin{equation}\label{zeta} \zeta_k=1+\mu_k\Delta. \end{equation}
Similarly we can write
\begin{equation}\label{factor2}x-\xi_i(1+o(1))=-{{1}\over{2\eta(\xi_i)}}[1+(\eta(\xi_i)+o(1))e^{-i\omega}][(1+(\eta(\xi_i)+o(1))e^{i\omega}],\end{equation}where
\begin{equation}\label{eta} \eta(\xi_i)=\xi_i-1\pm{\sqrt{(\xi_i-1)^2-1}},\end{equation}and the sign is chosen so that $|\eta(\xi_i)|< 1$.
Substituting (\ref{factor1}) and (\ref{factor2}) in (\ref{finalPSfMA}) immediately gives the corresponding asymptotic MA representation of $X_n$ of Theorem 2.1.
This completes the proof.
\halmos

\vskip .2in\noindent {\bf \large Proof of Theorem~\ref{clt}}\\
Without loss of generality, we assume that $\Delta<1$.  We assume also, as in Section 4, that $\sigma^2=1$.
Then the error of the innovations estimator, given by Corollary \ref{4.2}, is
\begin{equation*}
{\wh g}^\Delta(t)-g(t)=g^\Delta(t)-g(t)+\epsilon_n(t).
\end{equation*}
We multiply both sides by $\sqrt{n\Delta}$, obtaining
\begin{equation}\label{asym:ex:innov}\sqrt{n\Delta}({\wh g}^\Delta(t)-g(t))=\sqrt{n\Delta}(g^\Delta(t)-g(t))+\sqrt{n\Delta}\epsilon_n(t).\end{equation}
In order to prove our result, we need to ensure that, as $n\rightarrow\infty$ with $\Delta(n)$ satisfying the conditions specified in the statement of the theorem, (i) the first term on the right of \eqref{asym:ex:innov} converges to zero and (ii) the last term converges in distribution to a normal random variable with variance $\int_0^t g^2(u) du$.  The proofs follow.

(i)
Note first that
\begin{equation}\label{as:rep}
g^\Delta(t)-g(t)=\sum_{r=1}^p\left[C(r,\Delta)-\frac{b(\lambda_r)}{a'(\lambda_r)}\right]e^{\lambda_r t},
\end{equation}
where $C(r,\Delta)$ is the coefficient obtained plugging \eqref{psijdelta} and \eqref{sigmasquared} into \eqref{kernel:estimation:formula}.
The function $C(r,\Delta)$ is a rational bounded function, whose parameters depend continuously on $\Delta$.
Therefore we can write the series expansion
\begin{equation}\label{coef:series:exp}
C(r,\Delta)=\sum_{k=0}^\infty\bar{c}_k(r)\Delta^k
\end{equation}
where, from Theorem \ref{Woldapprox}(ii),  $\bar{c}_0(r)=b(\lambda_r)/a'(\lambda_r)$.
This implies that
$$C(r,1)= \bar{c}_0(r)+\sum_{k=1}^\infty \bar{c}_k(r)<\infty,\quad 1\leq r\leq p.$$
Then the deterministic part of \eqref{asym:ex:innov} can be written as
\begin{align*}\sqrt{n\Delta}\left|{g}^\Delta(t)-g(t)\right|&=\left|\sum_{r=1}^p\left(\sqrt{n}\sum_{k=1}^\infty \bar{c}_k(r)\Delta^{k+1/2}\right)e^{\lambda_r t}\right|.\\
&\leq\sum_{r=1}^p\left|\sqrt{n}\sum_{k=1}^\infty \bar{c}_k(r)\Delta^{k+1/2}\right|\left|e^{\lambda_r t}\right|.
\end{align*}
Since $\Delta<1$, for $k\geq 1$
$$0<\sqrt{n}\Delta^{k+1/2}= \Delta^{k-1} \sqrt{n}\Delta^{3/2} \leq \sqrt{n}\Delta^{3/2}.$$
Therefore,
$$0<\sqrt{n\Delta}|{g}^\Delta(t)-g(t)|\leq\Delta^{3/2}\sqrt{n}\sum_{r=1}^p\left|\sum_{k=1}^\infty \bar{c}_k(r)\right|\left|e^{\lambda_r t}\right|\rightarrow 0$$
if  $n(\Delta(n))^3\rightarrow 0$ as $n \rightarrow \infty$.

\vskip .1in
(ii) For fixed $\Delta>0$ Corollary \ref{4.2} implies that  $\sqrt{n\Delta}\epsilon_n(t)\Rightarrow N(0,a_\Delta\sigma_\Delta^2)$ as $n\rightarrow\infty$.
We shall show now that $\lim_{\Delta\rightarrow 0}a_\Delta\sigma_\Delta^2 = \int_0^t g^2(u) du$.
Then it follows that if $\Delta$ depends on $n$ in such a way that $\Delta(n)\rightarrow 0$ and $n\Delta(n)\rightarrow\infty$ as $n\rightarrow\infty$, we have the convergence in distribution,
$$\sqrt{n\Delta(n)}\epsilon_n(t)\Rightarrow N(0,\int_0^t g^2(u) du),~~{\rm as}~n\rightarrow\infty,$$
which, with (i), completes the proof of the theorem.

Now $a_\Delta\sigma_\Delta^2$ is the mean squared error of the best linear predictor of $Y_{\Delta\lfloor t/\Delta\rfloor}$  based on $Y_{k\Delta}, k=0,-1,-2,\ldots $,
and $\int_0^{\Delta\lfloor t/\Delta\rfloor} g^2(u)du$ is the mean squared error of the best linear predictor of $Y_{\Delta\lfloor t/\Delta\rfloor}$  based on $Y_t, t\in (-\infty,0]$.
The mean-square continuity of $Y$ means that the difference converges to zero as $\Delta\rightarrow 0$, which in turn implies that
$$a_\Delta\sigma_\Delta^2-\int_0^t g^2(u)du \rightarrow 0  ~~{\rm as}~\Delta\rightarrow 0.$$\vskip -.3in\halmos

\vskip .2in\noindent {\bf \large Proof of Theorem 6.2}

\noindent (a) \, If $Y$ is the CMA process (\ref{CMA}) then the spectral density of the sampled process $Y^\Delta$ is given (Bloomfield (2000), p.~196, Eq.~9.17) by
\begin{equation}\label{spec:sampling}
f_\Delta(\omega)=\frac1{\Delta}\sum_{k=-\infty}^\infty f_Y\Big(\frac{\omega+2k\pi}{\Delta}\Big),\quad \omega\in\Omega_d.
\end{equation}
Since $f_Y$ is positive, Eq.~\eqref{spec:sampling} can be rewritten as
\begin{equation}\label{spec:sampling2}
f_\Delta(\omega)=\Delta^{-1} f_Y(\Delta^{-1})\sum_{k=-\infty}^\infty\frac{f_Y(|\omega +2\pi k|\Delta^{-1})}{ f_Y(\Delta^{-1})},\quad \omega\in\Omega_d.
\end{equation}
Each of the summands converges by regular variation to $|\omega +2\pi k|^{-\al}$.
It remains to show that we can interchange the infinite sum with this limit.
Invoking the Potter bounds (Theorem~1.5.6 (iii) of \cite{BGT}),
for every $\epsilon>0$ there exists a $\Delta_\epsilon$,  such that for all $\Delta \leq \Delta_\epsilon$ and $|2\pi k+\omega|>0$
\begin{equation}\label{bound:s}
(1-\epsilon) |2\pi k+\omega|^{-\alpha-\epsilon}<\frac{f_Y(|\omega +2\pi k|\Delta^{-1})}{ f_Y(\Delta^{-1})}< (1+\epsilon) |2\pi k+\omega|^{-\alpha+\epsilon}.
\end{equation}
We take $\epsilon>0$ such that $\alpha-\epsilon>1$.  Then, using \eqref{bound:s}, we can bound \eqref{spec:sampling2} as follows:
\begin{equation}\label{bounds}(1-\epsilon)\frac{f_Y(\Delta^{-1})}{\Delta}\sum_{k=-\infty}^{\infty} |2\pi k+\omega|^{-\alpha-\epsilon}< f_\Delta(\omega)< (1+\epsilon) \frac{f_Y(\Delta^{-1})}{\Delta}\sum_{k=-\infty}^{\infty}|2\pi k+\omega|^{-\alpha+\epsilon},\quad \omega\in\Omega_d.\end{equation}
Since  $\epsilon$ can be chosen arbitrarily small, we conclude that as $\Delta\rightarrow 0$
$$f_\Delta(\omega)\sim \frac{f_Y(\Delta^{-1})}{\Delta}\sum_{k=-\infty}^\infty{|\omega+2k\pi |^{-\alpha}},\quad \omega\in\Omega_d.$$
We can rewrite the sum above as
\begin{align}\label{sum}
\sum_{k=-\infty}^\infty{|\omega+2k\pi |^{-\alpha}}
&= (2\pi)^{-\alpha}\sum_{k=-\infty}^\infty{\left|\frac{\omega}{2\pi}+k \right|^{-\alpha}}\notag\\
&= |\omega|^{-\alpha}+(2\pi)^{-\alpha} \sum_{k=0}^\infty\left[\left(k+1-\frac{\omega}{2\pi}\right)^{-\alpha}+\left(k+1+\frac{\omega}{2\pi}\right)^{-\alpha}\right],\quad \omega\in\Omega_d.
\end{align}
From this and the definition of  $\zeta$  we obtain \eqref{slowv:sample}.\\[2mm]
(b)\,  We first note that the Hurwitz zeta function $\zeta(-\alpha,1\pm\omega/2\pi)$ is bounded and strictly positive for all $\omega\in\Omega_d$, therefore, its integral over $[-\pi,\pi]$ is positive and finite.
{ On the other hand, since $\alpha>1$, the term $\omega^{-\alpha}$ is not integrable over $[-\pi,\pi]$.
However, the differenced} sequence $(1-B)^{\alpha/2 }Y^\Delta$,  has spectral density
\begin{equation}\label{filt:process}
h^\Delta(\omega)=2^{\alpha/2 } (1-\cos \omega)^{{\alpha/2 } }f_\Delta(\omega),\quad \omega\in\Omega_d.
\end{equation}
As $\Delta\rightarrow 0$ we can write, for $\omega\in\Omega_d$, by \eqref{slowv:sample}
\beao
h^\Delta(\omega)&\sim& 2^{\alpha/2 } (1-\cos \omega)^{{\alpha/2 } } \ell(\Delta^{-1})\Delta^{\alpha-1}\times\\
&&\left[|\omega|^{-\alpha}+(2\pi)^{-\alpha}\zeta\left(\alpha,1-\frac{\omega}{2\pi}\right)+
(2\pi)^{-\alpha}\zeta\left(\alpha,1+\frac{\omega}{2\pi}\right)\right].
\eeao The right hand side is integrable over $[-\pi,\pi]$ and
bounded in a neighbourhood of the origin, since
$2^{\alpha/2}(1-\cos \omega)^{\alpha/2
}\omega^{-\alpha}\rightarrow 1$ as $\omega\rightarrow0$.
Thus we conclude that the spectral density of the rescaled differenced sequence \eqref{rescaled} converges to that of a short-memory stationary process.\\[2mm]
(c) \,{ It is easy to check that the sampled $\CMA$ process has a Wold representation of the form (\ref{Wold1}) and that its one-step prediction mean-squared error based on the infinite past is $\sigma_\Delta^2$.  Kolmogorov's formula
(see, e.g., Theorem~5.8.1 of \cite{BD}) states that the one-step prediction mean-squared error for a discrete-time stationary process with spectral density $f$ is
\begin{equation}\label{Kolmogorov} \tau^2=2\pi\exp\left\{\frac{1}{2\pi}\int_{-\pi}^{\pi}\log f(\omega)d\omega\right\} \end{equation}
Applying it to the {\it differenced} process we find that its one-step prediction mean-squared error is
\beao
\lefteqn{2\pi\exp\left\{\frac{1}{2\pi}\int_{-\pi}^{\pi}\log h^\Delta(\omega)d\omega\right\}}\\
&=&2\pi\exp\left\{\frac{\alpha }{4\pi}\int_{-\pi}^{\pi}\log(2-2\cos\omega)d\omega\right\}\times \exp\left\{\frac{1}{2\pi}\int_{-\pi}^{\pi}\log f_\Delta(\omega)d\omega\right\} \, = \, \sigma^2_\Delta.
\eeao
Hence the differenced sequence has the same one-step prediction mean-squared error as $Y^\Delta$ itself.
Since from (\ref{bounds}), as $\Delta\rightarrow 0$,
$$\log f_\Delta(\omega)-\log(\ell(\Delta^{-1})\Delta^{\alpha-1})-\log\Big[ \sum_{-\infty}^\infty|2\pi k+\omega|^{-\alpha}\Big] \, \rightarrow 0 $$
pointwise on $\Omega_d$, and since the left side is dominated by an integrable function on $\Omega_d$, we conclude from the dominated convergence theorem that, as $\Delta\rightarrow 0$,
$${{1}\over{ \ell(\Delta^{-1})\Delta^{\alpha-1}}}\exp\left\{{{1}\over{2\pi}}\int_{-\pi}^\pi \log f_\Delta(\omega)d\omega\right\}\rightarrow \exp\left\{{{1}\over{2\pi}}\int_{-\pi}^\pi\log\left[\sum_{-\infty}^\infty |2\pi k+\omega|^{-\alpha}d\omega\right] \right\},$$
which, with (\ref{sum}) and  (\ref{Kolmogorov}), shows that as $\Delta\rightarrow 0$,
\begin{equation}\sigma^2_\Delta\sim 2\pi C_\alpha \ell\left(\Delta^{-1}\right) \Delta^{\alpha-1}.
\end{equation}}
This completes the proof.
\halmos

\end{document}